\documentclass[11pt,reqno]{amsart}

\usepackage{amsmath, amsfonts, amsthm, amssymb, mathtools}
\usepackage[usenames]{color}

\newtheorem{Theorem}{Theorem}[section]
\newtheorem{Lemma}{Lemma}[section]
\newtheorem{Proposition}{Proposition}[section]
\newtheorem{Corollary}{Corollary}[section]

\theoremstyle{definition}

\theoremstyle{remark}
\newtheorem{Remark}{Remark}[section]

\numberwithin{equation}{section}

\renewcommand{\r}{\rho}

\def\i{\varepsilon}

\renewcommand{\u}{{\bf u}}

\newcommand{\R}{{\mathbb R}}
\newcommand{\Dv}{{\rm div}}
\newcommand{\Cu}{{\rm curl}}

\newcommand{\dl}{\delta}

\def\f{\frac}
\renewcommand{\O}{\Omega}
\def\ov{\overline}

\def\D{\Delta }

\def\hf1{^\f{1}{1-\xi^2}}

\def\be{\begin{equation}}
\def\en{\end{equation}}
\def\bs{\begin{split}}
\def\es{\end{split}}

\newcommand{\F}{{\mathtt F}}

\newcommand{\supp}{{\rm supp}}

\author{Xianpeng Hu}
\address{Department of Mathematics, City University of Hong Kong, Hong Kong, PRC.} \email{xianpehu@cityu.edu.hk}

\title[Hausdorff dimension of concentration]
{Hausdorff dimension of concentration for isentropic compressible Navier-Stokes equations}

\keywords{Compressible Navier-Stokes equations, weak solutions, Hausdorff dimension}
\subjclass[2010]{35A01, 76N10, 35Q30.}
\date{\today}

\begin{document}

\begin{abstract}
The concentration phenomenon of the kinetic energy, $\r|\u|^2$, associated to isentropic compressible Navier-Stokes equations, is addressed in $\R^n$ with $n=2,3$ and the adiabatic constant $\gamma\in[1,\f{n}{2}]$. Except a space-time set with Hausdorff dimension less than or equal to $\Gamma(n)+1$ with $$
\Gamma(n)=\max\left\{\gamma(n), n-\f{n\gamma}{\gamma(n)+1}\right\}\quad\textrm{and}\quad\gamma(n)=\f{n(n-1)-n\gamma}{n-\gamma},$$ no concentration phenomenon occurs.
\end{abstract}

\maketitle

\section{Introduction}

In his pioneering work \cite{PL}, Lions establishes the global-in-time existence of weak solutions with finite energy for $n-$dimensional isentropic compressible Navier-Stokes equations
\begin{equation}\label{e1}
\begin{split}
&\partial_t\r+\Dv(\r\u)=0\\
&\partial_t(\r\u)+\Dv(\r\u\otimes\u)-\mu\D\u-\xi\nabla\Dv\u+\nabla\r^\gamma=0\\
&(\r,\u)|_{t=0}=(\r_0,\u_0)
\end{split}
\end{equation}
for $\gamma\ge\f32$ as $n=2$ and $\gamma\ge \f95$ as $n=3$, where $\mu>0$, $\mu+\xi>0$, the nonnegative scalar function $\r(x,t)$ stands for the density, and the vector-valued function $\u(x,t)\in \R^n$ is the velocity of the fluid. The weak solutions in \cite{PL} satisfy the energy law of \eqref{e1}
\begin{equation}\label{energy1}
\begin{split}
&E(\r(t),\u(t))+\int_0^t\int_{\R^n}\Big(\mu|\nabla\u|^2+\xi|\Dv\u|^2\Big)dxds\le E(\r_0,\u_0)
\end{split}
\end{equation}
with
\begin{equation*}
E(\r,\u)=
\begin{cases}
\f12\|\r|\u|^2\|_{L^1(\R^n)}+\|\r\|_{L^1(\R^n)}+\f{1}{\gamma-1}\|\r\|_{L^\gamma}^\gamma\quad&\textrm{if}\quad \gamma>1\\
\f12\|\r|\u|^2\|_{L^1(\R^n)}+\|\r\|_{L^1(\R^n)}+\|\r\ln\r\|_{L^1}\quad&\textrm{if}\quad \gamma=1.
\end{cases}
\end{equation*}
The existence result in \cite{PL} was later improved by Feireisl, Novotn\'{y} and Petzeltov\'{a} in \cite{FNP} to allow $\gamma>\f{n}{2}$ (see also \cite{BD, ER, F, NP}). One of the key ingredients in \cite{PL, FNP} is the higher integrability of the density due to the elliptic structure of the pressure $P(\r)=\r^\gamma$. Indeed, Lions used the elliptic structure of the pressure to improve the integrability of the density as
\begin{equation}\label{e1a}
\r(x,t)\in L_{loc}^q(\R^n\times\R^+)\quad\textrm{with}\quad q=\gamma+\f{2}{n}\gamma-1.
\end{equation}
The improved integrability \eqref{e1a} of $\rho$ fails as $\gamma\in[1,n/2]$. The energy estimate \eqref{energy1} implies that the concentration of the kinetic energy $\r|\u|^2$ never happens and the key issue for the construction of weak solutions for \eqref{e1} is the possible oscillation of the pressure term $P(\r)=\r^\gamma$ as $\gamma>n/2$. With the aid of \eqref{e1a}, the latter issue had been elegantly overcome by a convex argument in \cite{PL, FNP} in the framework of renormalised solutions. Since then whether those results in \cite{PL, FNP} still hold true for $\gamma\in[1,n/2]$ remains an outstanding open question which was listed among other unsolved problems of fluid mechanics in \cite{PL1}. 

As $\gamma\in[1,n/2]$, the main difficulty\footnote{As an example, quite different from the case as $\gamma>1$, the case $\gamma=1$ does not arise any oscillation issue for the pressure since it is a linear function of the density.} now turns to the possible concentration of the convective term, $\r\u\otimes\u$ and there naturally raises the question
\begin{equation}\tag{$\mathcal{Q}$}
\begin{split}
&\textit{How large the concentration set of the kinetic energy } \r|\u|^2\\
&\textit{is if only a priori energy estimate \eqref{energy1} is allowed?}
\end{split}
\end{equation}
This question is the core for the global-in-time existence of weak solutions with finite energy to \eqref{e1} as $\gamma\in [1,n/2]$, see \cite{SP, SP1, MP, MP1, PW}.  Note that as $\gamma>n/2$, the improved integrability \eqref{e1a} of the density and the \textit{a priori} energy estimate \eqref{energy1} ensure that $\r|\u|^2$ is bounded in $L^p_{loc}(\R^n\times\R^+)$ for some $p>1$, and this estimate prohibits the possibility of concentrations; while as $\gamma\in[1,n/2]$, energy estimates \eqref{energy1} imply that the kinetic energy $\r|\u|^2$ is merely bounded in $L^\infty(0,\infty; L^1(\R^n))$ which is not suitable for the weak convergence due to the non-reflexivity of $L^1$ and hence induces a possibility of a concentration phenomenon for the convective term $\r\u\otimes\u$.

A number of efforts have been paid on this subject for both the time-discretizied and the steady counterparts of \eqref{e1} as $\gamma\in[1,n/2]$. For the time-discretizied version of \eqref{e1}, the concentration-cancellation occurs in dimensions two (see section 6.6 in \cite{PL}), and in \cite{PS}, the concentration set has been verified to be $(\mathcal{H}^1, 1)$ rectifiable in dimensions three, where $\mathcal{H}^1$ stands for the one dimensional Hausdorff measure. For the steady version of \eqref{e1} as $\gamma=1$,  the concentration cancellation in dimensions two has been verified in the work \cite{FSW} due to a potential type estimate of the density (see Lemma 3.1 in \cite{FSW}), see also the most recent improvement along this direction for three dimensional cases in \cite{NP1, PW1}. However the size of the concentration set of \eqref{e1}, that is the question $\mathcal{Q}$, still remains an outstanding open question in multi-dimensional spaces.

In \cite{PL}, based on his concentration-compactness lemma, Lions introduced the weak star defect measure $\nu$ associated with the kinetic energy $\r|\u|^2$ for the time-discretized case to establish the concentration cancellation, but it is hard to apply this strategy to the time-dependent case \eqref{e1}. Indeed even through we know for the time-discretizied case in dimensions two, the concentration set of $\r|\u|^2$ is a countable union of delta mass (see Lemma 6.1 or Lemma 6.2 in \cite{PL}), what's the behaviour of the density away from these concentration points is still not clear. In this paper, we aim at the Hausdorff dimension of the concentration set of the kinetic energy $\r|\u|^2$ and also the property of the density on sets where the concentration does not occur. Moreover, as a byproduct, it will be verified that whenever the Hausdorff dimension of the concentration set is strictly less than one, no concentration occurs for the kinetic energy.

More precisely, this paper aims to provide a quantitative answer to the question ($\mathcal{Q}$) in $\R^n$ as $\gamma\in[1,n/2]$. One of our main tools in assessing the size of energy concentrations is the reduced defect measure $\theta$ associated with an arbitrary $H^1$ weakly convergent sequence $\u^\i (\u^\i\rightarrow\u)$ and a sequence of the nonnegative weights $\r^\i$ through the definition
$$\theta(E)=\limsup_{\i\rightarrow 0}\int_E\r^\i|\u^\i-\u|^2 dx$$
where $E$ is a Borel subset of $\R^n\times\R^+$. The set function $\theta$ is a nonnegative outer measure which vanishes precisely on these sets $E$ where $\sqrt{\r^\i}(\u^\i-\u)$ converges strongly to $0$; $\theta$ is concentrated on the $L^2$ exceptional sets of the sequences $\sqrt{\r^\i}(\u^\i-\u)$ where convergence is weak but not strong.

The first main result in this paper is an estimate for the Hausdorff dimension of the concentration set in terms of the reduced defect measure. In order to state this theorem, we recall that the weight which $\alpha-$order spherical Hausdorff measure assigns to a set is expressed as a limit of premeasures
$$\mathcal{H}^\alpha(E)=\lim_{r\rightarrow 0}\mathcal{H}^\alpha_r(E)=\sup_{r>0}\mathcal{H}^\alpha_r(E)$$
with the premeasure $\mathcal{H}_r^\alpha(E)$ given by
$$\mathcal{H}_r^\alpha(E)=\inf\Big\{\sum r_j^{\alpha}: E\subset \cup_{j=1}^\infty B_j, \quad 0<r_j\le r\Big\}.$$
Here $\{B_j\}$ is a countable covering of $E$ by open balls. The weight which $(\alpha, \beta)-$order cylindrical Hausdorff measure assigns to a space-time set is expressed as a limit of premeasures
$$\mathcal{H}^{\alpha, \beta}(E)=\lim_{r\rightarrow 0}\mathcal{H}^{\alpha,\beta}_r(E)=\sup_{r>0}\mathcal{H}^{\alpha,\beta}_r(E)$$
with the premeasure $\mathcal{H}^{\alpha, \beta}(E)$ being determined by the most efficient countable cover with cylinders $C(r_j, h_j)$ of height $h_j$ and sectional radius $r_j$:
$$\mathcal{H}_r^{\alpha, \beta}(E)=\inf\Big\{\sum r_j^\alpha h_j^\beta: E\subset \bigcup C(r_j, h_j),\quad 0\le r_j, h_j\le r\Big\}.$$ 
In general the cylindrical Hausdorff measure of order $(\alpha, \beta)$ is linked with the spherical Hausdorff measure of order $\alpha+\beta$ via
\begin{equation}\label{sc}
\mathcal{H}^{\alpha,\beta}(E)\le C(n)\mathcal{H}^{\alpha+\beta}(E).
\end{equation}
The main contribution of this paper reads as follows.
\begin{Theorem}\label{mt2}
Assume that $\gamma\in [1, n/2]$ and $(\r^\i,\u^\i)$ is a solution sequence (or an approximate solution sequence) of \eqref{e1} which satisfies the energy estimate \eqref{energy1} uniformly in $\i$. Then the reduced defect measure $\theta$ of a subsequence is concentrated inside a space-time set $E$ of Hausdorff dimension less than or equal to $\Gamma(n)+1$, i.e.
$$\mathcal{H}^{\Gamma(n)+\alpha, 1+\beta}(E)=0\quad\textrm{for all}\quad \alpha, \beta>0$$
with
$$\Gamma(n)=\max\left\{\gamma(n), n-\f{n\gamma}{\gamma(n)+1}\right\}\quad\textrm{and}\quad\gamma(n)=\f{n(n-1)-n\gamma}{n-\gamma}.$$
Moreover in the complement set of $E$, the weak limit $(\r,\u)$ of $(\r^\i,\u^\i)$ is a weak solution of \eqref{e1} in the sense of distributions and $\r^\i$ converges strongly to $\r$ in $L^1(E^c)$.
\end{Theorem}
\begin{Remark}
Since
$$\gamma(n)=n-\f{n}{n-\gamma}$$
and
$$n-\f{n\gamma}{\gamma(n)+1}=n-\f{n}{\f{n}{\gamma}-\f{1}{n-\gamma}},$$
the functions $\Gamma(n)$ and $\gamma(n)$ are strictly decreasing with respect to $\gamma$, and $n-2\le \gamma(n)\le\Gamma(n)<n-1$ as $\gamma\in [1,n/2]$. In particular, the cylindrical Hausdorff dimension $\Gamma(n)+1$ reads
\begin{equation*}
\Gamma(n)+1=
\begin{cases}
1 &\textrm{if}\quad n=2 \quad\textrm{and}\quad \gamma=1\\
2 &\textrm{if}\quad n=3 \quad\textrm{and}\quad \gamma=3/2\\
14/5 &\textrm{if}\quad n=3 \quad\textrm{and}\quad \gamma=1.
\end{cases}
\end{equation*}
\end{Remark}

Theorem \ref{mt2} is relied on an improved integrability of Riesz potential $I_1\r^\i$ of the density in $L^p(E^c)$ with $p>n$ and the space-time set $E^c$ is the collection of points where the maximal function of the density
$$\mathcal{M}\r(s,x)=\int_{|x-y|\le s}\r(y)dy$$ decays suitably faster, see Theorem \ref{4m}. 
In fact outside the exceptional set $E$, the maximal function of the density has the decay rate $\mathcal{M}\r(s,x)\le s^{\gamma(n)}$. Here the exponent $\gamma(n)$ is needed to improve the integrability of $I_1\r$ in $L^n$, see Lemma \ref{L52}, and the exceptional set $E$ is constructed by analysing the relation between the decay rate of the maximal function of the measure $\r dx$ and the integrability of the Bessel potential of that measure, see Theorem \ref{t3} and Theorem \ref{4m}. This idea is motivated by the work of DiPerna and Majda in \cite{DM, BM} for two dimensional incompressible Euler equations with vortex sheet initial data. The improved integrability of Riesz potential of the density in turn improves the integrability of the density itself, see Proposition \ref{p3}, outside the exceptional set $E$ due to the elliptic structure of the pressure in the momentum equation by adapting the argument in \cite{PL}. Except the elliptic structure of the pressure, the other two novel ingredients in improving the integrability of the density are the Div-Curl structure \eqref{51} where the concentration-cancellation phenomenon occurs and the improved estimate of Riesz potentials for momentum (see Lemma \ref{L51}). With the improved integrability of the density in hand, the classical convexity argument initially proposed by Lions in \cite{PL, FNP} produces the strong convergence of the density in the complement set of $E$. 

It is no surprise that the Riesz potential of the density, $I_1\r$, plays an important role in characterizing the concentration set of the kinetic energy. As an example, for the steady version of \eqref{e1} as $\gamma=1$, the entropy bound $\|\r\ln(1+\r)\|_{L^1}$ is not available, and only the total mass $\|\r\|_{L^1}$ can be assumed to be bounded. Under this situation, the Moser-Trudinger's inequality (see \cite{PL}) loses its power to control the concentration set in dimensions two. Fortunately, for the steady case, the density satisfies the following uniform potential estimate (see Lemma 3.1 in \cite{FSW})
\begin{equation}\label{pe}
\int_{B_r(x)}\f{\r}{|x-y|^\alpha}dy\le C(1+\|\r\|_{L^1}),
\end{equation}
which eliminates the concentration phenomena, where $\alpha>0$ and $B_r(x)$ stands for the ball with centre $x$ and radius $r$. The estimate \eqref{pe} indicates particularly that the density maximum function has the decay rate in the order $\alpha$, and hence there is no concentration for the steady isothermal compressible Navier-Stokes equations in dimensions two. Moreover the potential estimate \eqref{pe} also holds true both for the time-discretizied case (see Proposition 1 in \cite{PS}) and for the steady case \cite{NP1, PW1} in dimensions three. It is therefore natural to conjecture that the potential estimate of the type \eqref{pe} plays a similar role on characterizing concentration sets of the kinetic energy for the system \eqref{e1}. Furthermore notice that whenever $\gamma>n/2$, the \textit{a priori} estimate \eqref{energy1} implies that $\r^\gamma\in L^1$ and hence a natural decay for the density follows from H\"{o}lder's inequality
$$\mathcal{M}\r(x,s)=\int_{B_s(x)}\r(y)dy\le C\|\r\|_{L^{\gamma}(\R^n)} s^{n(\gamma-1)/\gamma},$$
which implies that no concentration occurs for $\gamma>n/2$ since $n(\gamma-1)/\gamma>\gamma(n)$.\footnote{Note also that $\gamma(n)=n(\gamma-1)/\gamma=n-2$ as $\gamma=n/2$ and as $\gamma\in[1,n/2)$, $n(\gamma-1)/\gamma<\gamma(n).$}

Due to the unidirectional relation \eqref{sc}, the other result in this paper concerns the spherical Hausdorff dimension of the reduced defect measure.
\begin{Theorem}\label{mt1}
If the reduced defect measure is concentrated inside a space-time set with spherical Hausdorff dimension less than one, then the reduced defect measure vanishes.
\end{Theorem}
The argument of Theorem \ref{mt1} in Section 5 is based on the uniform bound of the kinetic energy $\r|\u|^2$ and indicates that the concentration set has time coordinates of positive Lebesgue measure. The time axis is singled out because of the nature of the \textit{a priori} kinetic energy bound. This argument does not rely on the information of the density, and hence it is unable to deal with the oscillation and concentration issue for the pressure.


Let us end this introduction by emphasising that throughout this paper, only the energy estimate \eqref{energy1}, which contains the $L^1$ bound of the density or the conservation of total mass, is assumed, and the result shows the concentration only occurs in the set where the density concentrates very fast.  In other words, the pointwise estimate for the decay rate of the density maximal function implies a finer structure of the concentration set for the time-dependent system \eqref{e1}. 
  
The rest of this paper is organized as follows. In Section 2, the reduced defect measure is introduced and the relation between the reduced defect measure and the compactness is investigated. In Section 3, a uniformization estimate is established in order to make the uniform space exceptional set of the kinetic energy well defined. The improved estimate of the Riesz potential of the density in the complement of the space-time exceptional set is achieved in Section 4.
Finally, in Section 5 the integrability of the density away from the space-time exceptional set is improved and the proof of Theorem \ref{mt2} is demonstrated. Theorem \ref{mt1} is also verified in Section 5. Throughout this paper, the letters $\alpha, \beta,\lambda$ denote small but strictly positive parameters.

\bigskip\bigskip


\section{Reduced defect Measures and Compactness}

This section aims at the measurement of the concentration set of $\r^\i|\u^\i-\u|^2$ when both $\r^\i$, $\u^\i$, and $\u$ are functions of the spatial variables only and $\u^\i$ converges weakly to $\u$ in $H^1(\R^n)$.

Recall that in dimensions two for any $\r^\i\ge 0$ with $\|\r^\i\ln(1+\r^\i)\|_{L^1}<\infty$, there holds $\r^\i|\u^\i-\u|^2\in L^1_{loc}$ due to the inequality
\begin{equation}\label{20}
\r^\i|\u^\i-\u|^2\le \lambda\r^\i\ln(1+\r^\i)+\lambda\exp\left(\f{|\u^\i-\u|^2}{\lambda}-1\right)\quad\textrm{for all}\quad \lambda>0,
\end{equation}
and the Moser-Trudinger's inequality, $\exp(c|\u^\i-\u|^2)<\infty$, for some $c>0$ if $\u^\i-\u\in H^1(\O)$ with $\O\Subset\R^2$.
Based on the inequality \eqref{20}, Lions introduces the weak star defect measure in the framework of concentration-compactness which is the Radon measure $\nu$ so that
$$\r^\i|\u^\i-\u|^2\rightarrow \nu\quad \textrm{weak star},$$
and proves
(see Lemma 6.1 in \cite{PL}) that
\begin{equation}\label{2}
\nu=\sum_{i\in I}\nu_i\dl_{x_i}\quad \textrm{with} \quad\sum_{i\in I}\nu_i^{\f12}<\infty,
\end{equation}
where $I$ is an at most countable set, distinct points $\{x_i\}_{i\in I}\subset\R^2$, and positive constants $\{\nu_i\}_{i\in I}\subset\R^+.$ Moreover in terms of \eqref{20}, Lions's argument (Lemma 6.2 in \cite{PL}) shows that the concentration only occurs in the set where the embedding $H^1(\R^2)\mapsto e^{|\u^\i-\u|^2}$ concentrates.

In general, however, the vanishing of $\nu$ in \eqref{2} is not equivalent to the strong convergence of $\sqrt{\r^\i}(\u^\i-\u)$ in $L^2$. This observation motivates us to investigate the property of $\r^\i$ on the set where the weak star measure $\nu$ concentrates.

\subsection{Reduced defect measure and density maximal function}
We define the reduced defect measure $\theta$ as
$$\theta(E)=\limsup_{\i\rightarrow 0}\int_{E}\r^\i|\u^\i-\u|^2dx\quad \textrm{for every}\quad E\subset\R^n.$$
Generally speaking, due to its non-negativity, the density function $\r^\i$ in the definition of $\theta$ is treated as a weight of the integral, and the property of $\theta$ highly depends on the decay rate of the function $\r^\i$ (\cite{VM}).

An important consequence from this definition is
\begin{equation}\label{0a}
\theta(E)=0\Longleftrightarrow \sqrt{\r^\i}(\u^\i-\u)\rightarrow 0\quad\textrm{strongly in}\quad L^2(E).
\end{equation}
The partial relation between the weak star defect measure and the reduced defect measure is stated in the following proposition.
\begin{Proposition}
If $E$ is a closed set, then there hold
\begin{subequations}
\begin{align}
&\nu(E)=0\Longrightarrow \lim_{\i\rightarrow 0}\int_E\r^\i|\u^\i-\u|^2dx=0;\label{a}\\
&\nu(E)\ge \theta(E).\label{b}
\end{align}
\end{subequations}
\end{Proposition}
\begin{proof}
Equation \eqref{a} follows easily from relations \eqref{0a} and \eqref{b}, and we are left to prove \eqref{b}. Because $\nu$ is a Radon measure, $\nu$ is outer regular. Thus for any $\dl>0$, there exists an open set $G$ with $E\subset G$ and
$$\nu(E)+\dl\ge \nu(G).$$
By Urysohn's lemma, there exists $\psi\in C_0(\R^n)$, $0\le \psi\le 1$, and
\begin{equation}\label{u}
\psi(x)=
\begin{cases}
1, \quad x\in E\\
0, \quad x\in G^c
\end{cases}.
\end{equation}
There follows
\begin{equation*}
\begin{split}
\nu(E)+\dl&\ge \nu(G)\ge \int_G \psi d\nu\\
&=\lim_{\i\rightarrow 0}\int_G\psi\r^\i|\u^\i-\u|^2dx\\
&\ge\limsup_{\i\rightarrow 0}\int_E \psi\r^\i|\u^\i-\u|^2dx=\theta(E),
\end{split}
\end{equation*}
and letting $\dl\rightarrow 0$ yields the desired.
\end{proof}

We introduce the density maximal function $\mathcal{M}\r(s,x)$ of the density $\r(x)\ge 0$ as 
$$\mathcal{M}\r(s,x)=\int_{|x-y|\le s}\r(y)dy=\int_{B_s(x)}\r (y)dy,$$ where $B_s(x)$ stands for the ball with radius $s$ and centre at $x$. The density decays at every point of $\R^n$ with a power rate.

\begin{Proposition}\label{p1}
\begin{itemize}
\item For $\gamma>1$, the density maximal function $\mathcal{M}\r(s,x)$ has the decay rate
$$\mathcal{M}\r(s,x)\le C(n) \|\r\|_{L^\gamma(\R^n)}s^{n(\gamma-1)/\gamma}$$
for all $x\in\R^n$ and $s>0$.
\item For $\gamma=1$, the density maximal function $\mathcal{M}\r(s,x)$ has the decay rate
$$\mathcal{M}\r(s,x)\le C(n)(\ln s^{-1})^{-1}$$ for all $x\in\R^n$ and $s\in (0,1)$, where $C(n)$ depends only on the quantity $\|\r\ln(1+\r)\|_{L^1(\R^n)}$.
\end{itemize}
\end{Proposition}
\begin{proof}
The decay estimate for $\mathcal{M}\r(s,x)$ as $\gamma>1$ follows directly from H\"{o}lder's inequality. For $\gamma=1$,
set $\ln^{-}(s)=-\ln\min\{s,1\}$. For $s\in (0,1)$ we deduce from \eqref{20} that
\begin{equation*}
\begin{split}
(\ln s^{-1})\mathcal{M}\r(s,x)&\le \int_{|x-y|\le s}\ln^{-}|x-y| \r(y) dy\\
&\le \int_{|x-y|\le s}\r\ln(1+\r)dy+e^{-1}\int_{|x-y|\le s} e^{\ln^{-}|x-y|}dy\\
&\le \|\r\ln(1+\r)\|_{L^1}+e^{-1}\int_{|x-y|\le s}|x-y|^{-1}dy\\
&\le \|\r\ln(1+\r)\|_{L^1}+C(n) e^{-1}
\end{split}
\end{equation*}
as claimed.
\end{proof}

A slightly higher decay rate of the density maximal function yields the smoothness of the Newtonian potential of the density.
\begin{Lemma}
\begin{itemize}
\item For $n=2$, a slightly higher decay rate as
\begin{equation}\label{21a}
\mathcal{M}\r(s,x)\le (\ln s^{-1})^{-\beta} \quad\textrm{with}\quad \beta>1
\end{equation}
guarantees that
\begin{equation}\label{21}
\int_{|x-y|\le s}\ln^{-}|x-y| \r(y) dy\le \Big(1+1/(\beta-1)\Big)(\ln s^{-1})^{1-\beta}.
\end{equation}
\item For $n=3$, a slightly higher decay rate as
\begin{equation}\label{21aa}
\mathcal{M}\r(s,x)\le s^{\beta} \quad\textrm{with}\quad \beta>1
\end{equation}
also guarantees that
\begin{equation}\label{21aa1}
\int_{|x-y|\le s}|x-y|^{-1} \r(y) dy\le \Big(1+1/(\beta-1)\Big)s^{\beta-1}.
\end{equation}
\end{itemize}
\end{Lemma}
\begin{proof}
The arguments for $n=2,3$ are basically same, and we consider $n=2$ for instance.
For \eqref{21} integration by parts gives
\begin{equation*}
\begin{split}
\int_{|x-y|\le s}\ln^{-}|x-y| \r(y) dy&=\int_0^s \ln \sigma^{-1}d\mathcal{M}\r(\sigma,x)\\
&=(\ln \sigma^{-1})\mathcal{M}\r(\sigma, x)\Big|_{\sigma=0}^{\sigma=s}+\int_0^s\sigma^{-1} \mathcal{M}\r(\sigma,x) d\sigma\\
&\le (\ln s^{-1})^{1-\beta}+\int_0^s \sigma^{-1}(\ln\sigma^{-1})^{-\beta}d\sigma\\
&= (\ln s^{-1})^{1-\beta}+\int_{\ln s^{-1}}^\infty \tau^{-\beta}d\tau\\
&\le \Big(1+1/(\beta-1)\Big)(\ln s^{-1})^{1-\beta}.
\end{split}
\end{equation*}
\end{proof}

\begin{Remark}
The H\"{o}lder-type inequality (see \cite{ER}) gives
$$ab\le a\ln^\beta(1+a)+\exp(b^{1/\beta})\quad \textrm{for all}\quad a\ge 0,$$
where $\beta>1$, and thus the higher decay rate of the density maximal function \eqref{21a} holds true whenever $\|\r\ln^\beta(1+\r)\|_{L^1}$ is finite. The higher decay rate of the density maximal function as \eqref{21a} plays an important role in the vanishing of the reduced defect measure. As we will see later on that for every set where the decay rate \eqref{21a} holds the concentration never happens.
\end{Remark}

\begin{Remark} The decay rate as $\gamma=1$ in Proposition \ref{p1} is optimal in dimensions two. For instance the function $\r=\f{m^2}{\ln m}1_{B_{1/m}}$ has the exact decay rate 
$$\mathcal{M}\r(1/m,0)=(\ln m)^{-1}.$$
\end{Remark}

\subsection{Compactness}

The objective of this subsection is to investigate the decay rates of the density maximal function that allow the weighted $L^2(\r dx)$ norm of a function to be estimated by the $L^p$ norm of its partial derivatives as $p\in (1,2]$. Let $\mu=\r dx$ be a finite nonnegative Radon measure and $\mu_K(E)=\mu(E\cap K)$ for every Borel set $E$ and compact set $K\Subset \R^n$. Maz'ya established \cite{VM} that the embedding $H^{1}\mapsto L^2(d\mu)$ is continuous if and only if $\mu(E)\le c Cap(E)$ for all Borel set $E\Subset R^n$. Moreover, Admas \cite{AD} proved that the embedding is continuous if $\|(-\D)^{-\f12}\mu_K\|_{L^2}^2\le C\mu(K)$. Note that formally we have
$$\|(-\D)^{-\f12}\mu_K\|_{L^2}^2=-\int_{\R^n}\D^{-1}\mu_Kd\mu_K\le \sup_{\R^n}|\D^{-1}\mu_K|\mu(K).$$
Motivated by the inequality above, it is expected that the $L^\infty$ bound of the potential $(-\D)^{-1}\mu=\int G(x-y)d\mu(y)$, where $G(x)$ is the fundamental solution of $-\D$ in $\R^n$, ensures the continuity of the embedding, and a smoothness condition of the potential guarantees the compactness of the embedding. Indeed the decay rate as \eqref{21a} or \eqref{21aa} is good enough to ensure the continuity and the compactness of the embedding $H^1\mapsto L^2(\r dx)$.

\begin{Proposition}\label{p2}
For compact subsets $K\Subset\R^n$, there exists a constant $c_0$ depending only on $K$ and $p$ such that for all $\u\in W_0^{1,p}(K)$ with $p\in (\f{2n}{n+1},2]$
$$\int_K|\u|^2\r dx\le c_0\left(\omega_p(s)^{2(p-1)/p}\|\nabla\u\|_{L^{p}(K)}^2+(1+s^{-2n})^2\Big(\|\u\|^{2p}_{L^p(K)}+\|\u\|_{L^1(K)}^{2/p}\Big)\right),$$
where
$$\omega_p(s)=\begin{cases}
\sup_{x\in K} \int_{|x-y|\le s}G(x-y)\r(y)dy\quad&\textrm{if}\quad p=2\\
\sup_{x\in K} \int_{|x-y|\le s}|x-y|^{(p-n)/(p-1)}\r(y)dy\quad&\textrm{if}\quad p\in (\f{2n}{n+1},2),
\end{cases}$$
and $G(x)$ is the fundamental solution of $-\D$ in $\R^n$.
\end{Proposition}
\begin{proof}
Set the quadratic form $\mathcal{L}: W^{1,p}(E)\mapsto \R$ by
$$\mathcal{L}(v)=\omega_p(s)^{p-1}\|\nabla v\|^{p}_{L^p(E)}+(1+s^{-2n})^{p}\|v\|_{L^1(E)}^{p}$$
and
$$Cap_{\mathcal{L}}(E):=\inf\{\mathcal{L}(v): v\in C_0^\infty(\R^n)\quad\textrm{and}\quad v(x)\ge 1\quad \textrm{for all} \quad x\in E\}.$$

\texttt{Step 1.} We claim that there is a constant $c$ depending only on $K$ such that 
$$\|\r\|_{L^1(E\cap K)}\le c \Big(Cap_{\mathcal{L}}(E)\Big)^{2/p}\qquad \textrm{for all Borel sets } E.$$

To this end, we notice that for any $v\in W^{1,p}_0(K)$
$$v(x)=\f{1}{n\varpi_n}\int_{K}\f{x-y}{|x-y|^n}\cdot\nabla v(y)dy,$$where $\varpi_n$ denotes the volume of the unit ball in $\R^n$.
For an increasing function $\theta(\sigma)\in C^\infty(\R)$ such that $\theta(\sigma)=0$ for $\sigma<s/8$, $\theta(\sigma)=1$ for $\sigma>s/4$ and $\theta'(\sigma)<16/s$, one has
\begin{subequations}\label{22}
\begin{align}
v(x)&=\f{1}{n\varpi_n}\int_{K}(1-\theta(|x-y|))\f{x-y}{|x-y|^n}\cdot\nabla v(y)dy\label{22a}\\
&\quad+\f{1}{n\varpi_n}\int_{K}\theta(|x-y|)\f{x-y}{|x-y|^n}\cdot\nabla v(y)dy.\label{22b}
\end{align}
\end{subequations}

For \eqref{22a}, one has
\begin{equation*}
\begin{split}
\int_{E\cap K} |\eqref{22a}|\r(x)dx &\le c\int_{E\cap K} g\star |\nabla v|(x) \r(x)dx\\
&=c\int_{\R^n}|\nabla v|(y)\int_{E\cap K}g(y-x)\r(x)dx dy\\
&\le c\|\nabla v\|_{L^p}\left(\int_{\R^n}\left(\int_{E\cap K} g(y-x)\r(y) dy\right)^{p/(p-1)}dx\right)^{(p-1)/p}
\end{split}
\end{equation*}
where $g(x)=|x|^{1-n}$ for $|x|\le s/4$ and $g(x)=0$ otherwise. Due to Lemma \ref{la1} one knows
$g^{\f{p}{2(p-1)}}\star g^{\f{p}{2(p-1)}}(x)\le \ov{G}(x)$ vanishes for $s\le |x|\le 1$, and 
for $1<p\le 2$, one has
\begin{equation*}
\begin{split}
\left(\int_{E\cap K}g(y-x)\r(y)dy\right)^{\f{p}{2(p-1)}}&\le \left(\int_{E\cap K} g^{\f{p}{2(p-1)}}(x-y)\r(y)dy\right)\left(\int_{E\cap K}\r(y)dy\right)^{\f{2-p}{2(p-1)}}\\
&= \|\r\|_{L^1(E\cap K)}^{\f{2-p}{2(p-1)}}\int_{E\cap K} g^{\f{p}{2(p-1)}}(x-y)\r(y)dy.
\end{split}
\end{equation*}
Therefore
\begin{equation*}
\begin{split}
&\int_{\R^n}\left(\int_{E\cap K} g(y-x)\r(y)dy\right)^{p/(p-1)}dx\\
&\quad\le  \|\r\|_{L^1(E\cap K)}^{\f{2-p}{p-1}}\int_{\R^n}\left(\int_{E\cap K} g^{\f{p}{2(p-1)}}(y-x)\r(y)dy\right)^{2}dx\\
&\quad= \|\r\|_{L^1(E\cap K)}^{\f{2-p}{p-1}}\int_{E\cap K}\left(\int_{E\cap K}g^{\f{p}{2(p-1)}}\star g^{\f{p}{2(p-1)}}(y-x)\r (y)dy\right)\r(x)dx\\
&\quad\le  \|\r\|_{L^1(E\cap K)}^{\f{2-p}{p-1}}\int_{E\cap K}\left(\int_{|x-y|<s}\ov{G}(x-y)\r(y)dy\right)\r(x)dx\\
&\quad\le \omega_p(s)\|\r\|_{L^1({E\cap K})}^{\f{1}{p-1}},
\end{split}
\end{equation*}
which yields
\begin{equation}\label{22aa}
\begin{split}
\int_{E\cap K}|\eqref{22a}|\r(x)dx&\le \|\nabla v\|_{L^p}\omega_p(s)^{(p-1)/p}\|\r\|_{L^1({E\cap K})}^{1/p}\\
&\le C(p)\|\nabla v\|_{L^p}\omega_p(s)^{(p-1)/p}\|\r\|_{L^1({E\cap K})}^{1/2},
\end{split}
\end{equation}
due to $\|\r\|_{L^1(\R^n)}<\infty$ and $p\in(1,2]$.
Since the function $\theta(|y|)|y|^{-n}y$ is smooth and the norm of its gradient is bounded by $cs^{-n}$, one has
$$|\eqref{22b}|\le cs^{-n}\|v\|_{L^1(K)},$$ which implies
\begin{equation*}
\int_{E\cap K}|\eqref{22b}|\r(x)dx\le cs^{-n}\|v\|_{L^1(K)}\|\r\|_{L^1({E\cap K})}\le cs^{-n}\|v\|_{L^1(K)}\|\r\|_{L^1({E\cap K})}^{1/2}
\end{equation*}
since $\|\r\|_{L^1(\R^n)}<\infty$.
This, combining with \eqref{22aa}, yields
$$\int_{E\cap K}|v|\r dx\le c\mathcal{L}(v)^{1/p}\|\r\|_{L^1({E\cap K})}^{1/2}.$$
If $v\ge 1$ on $E$, there follows
$$\|\r\|_{L^1({E\cap K})}\le\int_{E\cap K}|v| \r dx\le c\mathcal{L}(v)^{1/p}\|\r\|_{L^1({E\cap K})}^{1/2},$$
which yields the claim.

\texttt{Step 2.} There exists a positive constant $c$, which depends only on $K, p$,  but does not depend on $s$, such that the inequality
\begin{equation*}
\begin{split}
\int_0^\infty \Big[Cap_{\mathcal{L}}(N_\lambda)\Big]^{2/p}d\lambda^{4/p}\le c\omega_p(s)^{2(p-1)/p}\|\nabla v\|_{L^p(K)}^2+c(1+s^{-2n})^2(\|v\|_{L^1(K)}^{2/p}+\|v\|^{2p}_{L^p(K)})
\end{split}
\end{equation*}
holds for every function $v\in C_0^\infty$, where $N_\lambda=\{x\in K: |v(x)|\ge \lambda^{2/p}\}.$

Since $Cap_{\mathcal{L}}(N_\lambda)$ decreases in $\lambda$, there holds
\begin{equation*}
\int_0^\infty \Big[Cap_{\mathcal{L}}(N_\lambda)\Big]^{2/p}d\lambda^{4/p}\le (2^{4/p}-1)\sum_{j=-\infty}^{j=\infty} 2^{4j/p} \Big[Cap_{\mathcal{L}}(N_{2^j})\Big]^{2/p}.
\end{equation*}
Let $\eta\in C^\infty(\R)$ be a non-decreasing function such that $\eta(s)=0$ for $s\le 0$, $\eta(s)=1$ for $s\ge 2^{2/p}-1$ and $\eta'(s)\le 2(2^{2/p}-1)^{-1}$ for all $s$. Denote $v_j=\eta(2^{2(1-j)/p}|v(x)|-1)$. Since $v_j=0$ whenever $|v(x)|\le 2^{2(j-1)/p}$ and $v_j(x)=1$ as $|v(x)|\ge 2^{2j/p}$. The definition of the capacity $Cap_{\mathcal{L}}$ gives
\begin{equation*}
\begin{split}
Cap_{\mathcal{L}}(N_{2^j})\le \omega_p(s)^{p-1}\int_{N_{2^{j-1}}\setminus N_{2^j}}|\nabla v_j|^pdx+(1+s^{-2n})^p\left(\int_{N_{2^{j-1}}}|v_j|dx\right)^p.
\end{split}
\end{equation*}
Note that $|\nabla v_j|\le c(p)2^{2(1-j)/p}|\nabla v|$ and $|v_j|\le 1$. Moreover if $x\in N_{2^{j-1}}$, then one has $|v_j(x)|\le 2^{2(1-j)/p}|v(x)|$. Therefore for $j\le 0$ one has
\begin{equation}\label{2201}
Cap_{\mathcal{L}}(N_{2^j})\le c(p)2^{2(1-j)}\omega_p(s)^{p-1}\int_{N_{2^{j-1}}\setminus N_{2^j}}|\nabla v|^pdx+(1+s^{-2n})^p2^{2(1-j)/p}|K|^{p-1}\int_{N_{2^{j-1}}}|v|dx,
\end{equation}
while for $j\ge 1$ one has
\begin{equation}\label{2202}
Cap_{\mathcal{L}}(N_{2^j})\le c(p)2^{2(1-j)}\omega_p(s)^{p-1}\int_{N_{2^{j-1}}\setminus N_{2^j}}|\nabla v|^pdx+(1+s^{-2n})^p2^{2(1-j)}\left(\int_{N_{2^{j-1}}}|v|dx\right)^p.
\end{equation}
Denote by $1_E$ the indicator function of the set $E$. Since $p\le 2$, Minkowski's inequality yields
\begin{equation*}
\begin{split}
\sum_{j=-\infty}^\infty\left(\int_{N_{2^{j-1}}\setminus N_{2^j}}|\nabla v|^pdx\right)^{2/p}&=\sum_{j=-\infty}^\infty\left(\int_{K}1_{N_{2^{j-1}}\setminus N_{2^j}}|\nabla v|^pdx\right)^{2/p}\\
&\le \left(\int_{K}\left(\sum_{j=-\infty}^\infty1_{N_{2^{j-1}}\setminus N_{2^j}}|\nabla v|^2\right)^{p/2}dx\right)^{2/p}\\
&\le \left(\int_{K}|\nabla v|^{p}dx\right)^{2/p}.
\end{split}
\end{equation*}
Therefore one deduces from \eqref{2201} and \eqref{2202} that
\begin{equation*}
\begin{split}
\int_0^\infty \Big[Cap_{\mathcal{L}}(N_\lambda)\Big]^{2/p} d\lambda^{4/p}&\le c(p)\omega_p(s)^{2(p-1)/p}\left(\int_{K}|\nabla v|^pdx\right)^{2/p}+c(K)(1+s^{-2n})^2\|v\|_{L^1(K)}^{2/p}\\
&\quad+c(1+s^{-2n})^2\sum_{j\ge 0}\left(\int_{N_{2^{j-1}}}|v|dx\right)^2.
\end{split}
\end{equation*}
It is easy to see that
$$\left(\int_{N_{2^{j-1}}}|v|dx\right)^2\le |N_{2^{j-1}}|^{2(p-1)/p}\left(\int_K|v|^pdx\right)^{2/p}\quad\textrm{and}\quad |N_{2^{j-1}}|\le 2^{2(1-j)}\int_K |v|^pdx,$$
Thus
$$\left(\int_{N_{2^{j-1}}}|v|dx\right)^2\le 2^{4(p-1)(1-j)/p}\|v\|_{L^p(K)}^{2p},$$
which yields the desired.

\texttt{Step 3.} We use Step 1 to obtain $\|\r\|_{L^1(K\cap E)}\le c\left(Cap_{\mathcal{L}}(E)\right)^{2/p}$ for every Borel set $E$. Thus there follows that
$$\int_K |v|^2\r dx=\int_0^\infty \|\r\|_{L^1(K\cap N_\lambda)}d\lambda^{4/p}\le c\int_0^\infty \Big[Cap_{\mathcal{L}}(N_\lambda)\Big]^{2/p}d\lambda^{4/p}.$$
This, combining Step 2, gives the desired.

\end{proof}

Proposition \ref{p2} implies that the bound of $\omega_p(s)$ gives rise to the continuity of the embedding $W^{1,p}\mapsto L^2(\r dx)$. Actually a little bit smoothness of $\omega_p(s)$ will produce the compactness of that embedding.

\begin{Corollary}\label{t2}
Let $E$ be the set of points with property \eqref{21a} or \eqref{21aa}. Then for any compact set $K\Subset\R^n$ and for any $\eta>0$ there exists a positive constant $c(\eta)$, which depends on $K, \eta$ and does not depend on $v$, such that the inequality
\begin{equation}\label{c1}
\int_{K\cap E}\r|v|^2dx\le \eta\|v\|_{H^{1}}^2+c(\eta)(\|v\|_{L^2}^4+\|v\|_{L^1(K)})
\end{equation}
holds for all $v\in H_0^{1}(K)$. 
\end{Corollary}
\begin{proof}
The estimate \eqref{21} or \eqref{21aa1} implies that $\omega_2(s)\rightarrow 0$ as $s\rightarrow 0$, and hence the inequality \eqref{c1} follows from Proposition \ref{p2} by setting $s$ to be sufficiently small so that $\omega_2(s)\le \eta$.
\end{proof}

Corollary \ref{t2} implies that a slightly higher decay rate of the density maximal function prevents the concentration, and hence Corollary \ref{t2}, combining with \eqref{2} and Proposition \ref{p1}, implies the following corollary.
\begin{Corollary}
The reduced defect measure $\theta$ in $\R^2$ concentrates on a subset of countable points $\{x_i\}$, whose density maximal functions at $x_i$ violate the decay rate \eqref{21a}.
\end{Corollary}

Actually the set $E$ of points with higher decay rate \eqref{21a} and \eqref{21aa} yields a kind of strong convergence.

\begin{Proposition}\label{p6}
Let $E$ be the set of points with property \eqref{21a} or \eqref{21aa} uniformly for $\r^\i$ and $\r$. For any compact $K\Subset\R^n$ and $\phi\in C(\R^n)$, 
the following statements hold true
\begin{itemize}
\item For $f,g\in H^{1}(\R^n)$,
$$\lim_{\i\rightarrow 0}\int_{K\cap E}\phi(\r^\i-\r)fg dx=0.$$
\item $\lim_{\i\rightarrow 0}\left|\int_{K\cap E}\phi(\r^\i|\u^\i|^2-\r|\u|^2)dx\right|=0.$
\item $\lim_{\i\rightarrow 0}\left|\int_{K\cap E}\phi(\r^\i\u^\i-\r\u)dx\right|=0.$
\end{itemize}
\end{Proposition}
\begin{proof}
For the first one, fix $\dl>0$ and choose $\ov{f}, \ov{g}\in C_0^\infty(\R^n)$ such that
$$\|f-\ov{f}\|_{H^1}+\|g-\ov{g}\|_{H^1}\le \dl.$$
Thus one deduces from Corollary \ref{t2} that
\begin{equation*}
\begin{split}
\limsup_{\i\rightarrow 0}\left|\int_{K\cap E}\phi(\r^\i-\r)fgdx\right|&\le \limsup_{\i\rightarrow 0}\left|\int_{K\cap E}\phi(\r^\i-\r)\ov{f}\ov{g} dx\right|\\
&\quad+c(\phi)\int_{K\cap E}(\r^\i+\r)(|g||f-\ov{f}|+|f||g-\ov{g}|)dx\\
&\le c\dl (\|f\|_{H^1}+\|g\|_{H^1})\rightarrow 0\quad\textrm{as}\quad \dl\rightarrow 0,
\end{split}
\end{equation*}
which yields the desired.

For the second one, one has, using the first identity
\begin{equation*}
\begin{split}
&\lim_{\i\rightarrow 0}\left|\int_{K\cap E}\phi(\r^\i|\u^\i|^2-\r|\u|^2)dx\right|\\
&\quad=\lim_{\i\rightarrow 0}\left|\int_{K\cap E}\phi\r^\i(|\u^\i|^2-|\u|^2)dx\right|\\
&\quad\le c(\phi)\limsup_{\i\rightarrow 0}\left(\f{1}{\sqrt{\eta}}\int_{K\cap E}\r^\i|\u^\i-\u|^2dx+\sqrt{\eta}\int_{K\cap E}\r^\i|\u^\i+\u|^2dx\right)\\
&\quad\le c(\phi)\limsup_{\i\rightarrow 0}\Big(\sqrt{\eta}\|\u^\i-\u\|_{H^1}+c(\eta)\|\u^\i-\u\|_{L^2}+\sqrt{\eta}\Big)\\
&\quad\rightarrow 0 \quad\textrm{as}\quad \eta\rightarrow 0
\end{split}
\end{equation*}
as claimed.

For the third one,  since $\r^\i|\u^\i-\u|\le C(\eta)\r^\i|\u^\i-\u|^2+\eta \r^\i,$ one has
\begin{equation*}
\begin{split}
\int_{K\cap E}\r^\i|\u^\i-\u|dx\le \eta \|\r^\i\|_{L^1}+C(\eta)\int_{K\cap E}\r^\i|\u^\i-\u|^2dx.
\end{split}
\end{equation*}
The first term approaches to zero as $\eta\rightarrow 0$ due to the uniform mass, while the second term approaches to zero due to the second claim. Therefore, there follows
$$\lim_{\i\rightarrow 0}\int_{K\cap E}\r^\i|\u^\i-\u|dx=0$$
and hence
$$\lim_{\i\rightarrow 0}\left|\int_{K\cap E}(\r^\i\u^\i-\r\u)dx\right|=0.$$
\end{proof}

\bigskip\bigskip


\section{Uniformization Estimates}

In this section we consider the decay rate of the density maximal functions associated with an approximate sequence of the density $\r^\i$, which characteristics the concentration set. Throughout this section, the nonnegative density $\r$ depends only on the spatial variable $x$ and satisfies a uniform bound
\begin{equation}\label{energy3}
C\ge
\begin{cases}
\|\r\|_{L^1(\R^n)}+\|\r\|_{L^\gamma}\quad&\textrm{if}\quad \gamma>1\\
\|\r\|_{L^1(\R^n)}+\|\r\ln(1+\r)\|_{L^1}\quad&\textrm{if}\quad \gamma=1,
\end{cases}
\end{equation}
for some universal positive constant $C$.

\begin{Proposition}\label{p}
Suppose that $\r$ is a nonnegative measure with finite total mass; that is $\|\r\|_{L^1(\R^n)}<\infty$. Then for $r\in (0,1)$ and an arbitrary small positive constant $\alpha$, the closed set
$$E_r=\{x\in\R^n: \mathcal{M}\r(s,x)\le  s^{\gamma(n)+\alpha},\quad 0\le s\le r\}$$
satisfies
$$\mathcal{H}_r^{\gamma(n)+\alpha}(E_r^c)\le 5^{\gamma(n)+\alpha}\|\r\|_{L^1}.$$
\end{Proposition}

Note that $\{E_r\}_{\{r\ge0\}}$ is the family of ``good" sets in terms of the decay of the radial distribution $\mathcal{M}\r(s,x)$; $\{E_r^c\}$ is the family of ``bad" sets. The size of the exceptional set, $E_r^c$, is bounded in terms of the total mass of $\r$ and a constant which depends only on $\gamma, n$.

\begin{proof}[Proof of Proposition \ref{p}]
The fact that the set $E_r$ is closed is a direct consequence of the continuity of integral.
Moreover, by hypothesis we have
$$E_r^c=\{x\in\R^n: \mathcal{M}\r(s,x)>s^{\gamma(n)+\alpha}\quad\textrm{for some s},\quad 0\le s\le r\}.$$

We observe that $E_r^c\subset \bigcup_{x\in E_r^c}B_{s(x)}(x)$. Let $\{B_j\}$, $B_j=B_{s_j(x_j)}(x_j)$ be the subcovering described in the covering lemma (see Theorem 1.3.1 in \cite{WP}), then
$$E^c_r\subset \bigcup_{j=1}^\infty B_j^*\quad\textrm{with}\quad B_j^*=B_{5s_j}(x_j).$$
Thus
\begin{equation*}
\begin{split}
\mathcal{H}_r^{\gamma(n)+\alpha}(E_r^c)&\le \sum_{j=1}^\infty (5s_j)^{\gamma(n)+\alpha}\le 5^{\gamma(n)+\alpha} \sum_{j=1}^\infty \mathcal{M}\r(s_j, x_j)\\
&=5^{\gamma(n)+\alpha}\int_{\cup_{j=1}^\infty B_j}\r(y)dy\le 5^{\gamma(n)+\alpha} \int_{\R^n}\r(x)dx.
\end{split}
\end{equation*} 
\end{proof}

The uniform bound \eqref{energy3} yields the $L^{n\gamma/(n-\gamma)}$ estimate of the Riesz potential associated with the density. That is
\begin{Lemma}\label{L31}
For the nonnegative density $\r$ with the uniform bound \eqref{energy3}, there holds true that
\begin{equation*}
\begin{cases}
\|I_1\r\|_{L^{n/(n-1)}(B_R)}\le C(R)\Big(1+\|\r\|_{L^1(\R^n)}+\|\r\ln(1+\r)\|_{L^1(\R^n)}\Big) \quad&\textrm{if}\quad \gamma=1,\\
\|I_1\r\|_{L^{n\gamma/(n-\gamma)}(\R^n)}\le C\|\r\|_{L^\gamma(\R^n)}\quad&\textrm{if}\quad \gamma\in (1,n/2],
\end{cases}
\end{equation*}
where $B_R$ stands for the ball with radius $R$ and centre $0$.
\end{Lemma}
\begin{proof}
The inequality for $\gamma>1$ follows from the classical estimate of the Riesz potential (\cite{ST}). 

Set
\begin{equation}\label{L31a}
\mathcal{M}_\beta\r(x)=\sup_{r>0}r^{\beta-n}\int_{B_r(x)}\r(y)dy.
\end{equation}
The finite mass $\|\r\|_{L^1(\R^n)}$ implies that the function $\mathcal{M}_n\r(x)$ belongs to $L^\infty$. 
Moreover the finite entropy bound $\|\r\ln(1+\r)\|_{L^1(\R^n)}$ implies that the maximal function $\mathcal{M}_0\r(x)$ belongs to $L^1(B_R)$ for any $R>0$ (see Page 23 in \cite{ST})
\begin{equation}\label{30b2}
\begin{split}
\|\mathcal{M}_0\r(x)\|_{L^1(B_R)}\le C(R)(1+\|\r\ln(1+\r)\|_{L^1(\R^n)}).
\end{split}
\end{equation}
These two uniform bounds, combining Proposition 3.2 in \cite{AD1}, yields that $I_1\r\in L^{n/(n-1)}(B_R)$.
\end{proof}

Lemma \ref{L31} yields a uniform bound of $I_1\r$ in $L^{n\gamma/(n-\gamma)}$. Note that $n\gamma/(n-\gamma)\le n$ for $\gamma\in [1,n/2]$.
Next we are going to improve the uniform bound of the Riesz potential $I_1\r$ in $L^p$ with $p>n$ for all $\gamma\in[1,n/2]$ when the density maximal function decays faster. This estimate is crucial in Section 5 for improving the integrability of the density itself. For this purpose, we denote by $1_\O$ the indicator function of $\O$; that is
\begin{equation*}
1_\O=
\begin{cases}
1\quad\textrm{if}\quad x\in \O\\
0\quad\textrm{if}\quad x\notin \O.
\end{cases}
\end{equation*} 
\begin{Lemma}\label{L52}
Let $\alpha\in(0,1)$ satisfy
\begin{equation}\label{300}
\gamma(n)-n+1+\alpha<0,
\end{equation} 
and $$\O=\{x\subset B_R: \mathcal{M}\r(s,x)\le s^{\gamma(n)+\alpha}\quad\textrm{for}\quad s\ge 0\}.$$ Then there holds
$$\|I_1(\r1_\O)\|_{L^p(\R^n)}\le C(\gamma,n, \alpha)\quad\textrm{with}\quad p=\gamma\f{\gamma(n)-n+\alpha}{\gamma(n)-n+1+\alpha}>n.$$
\end{Lemma}
\begin{proof}
First of all, we claim that for all $x\in\R^n$, one has
\begin{equation}\label{300a1}
\mathcal{M}(\r 1_\O)(s,x)\le (2s)^{\gamma(n)+\alpha}\quad\textrm{for}\quad s\ge 0.
\end{equation}
Indeed, if $x\in \O$, the claim is obviously true. Let $x\notin\O$. Note that $\O$ is a closed set and we set $d_x=\textrm{dist}\{x,\O\}$. If $s\le d_x$, then
$$\mathcal{M}(\r 1_\O)(s,x)\le \int_{|x-y|<s}(\r 1_\O)(y)dy=0\le (2s)^{\gamma(n)+\alpha}.$$
If $s>d_x$, then there exists $x_0\in \O\cap B_s(x)$ and $B_s(x)\subset B_{2s}(x_0)$. Thus
\begin{equation*}
\begin{split}
\mathcal{M}(\r 1_\O)(s,x)&\le \int_{|x-y|<s}(\r 1_\O)(y)dy\le \int_{|x-y|<s}\r(y)dy\\
&\le \int_{|x_0-y|<2s}\r(y)dy\le (2s)^{\gamma(n)+\alpha}
\end{split}
\end{equation*}
as claimed.

Next we follow the basic ideas of Adams \cite{AD1}.
For $\r\neq 0$, set
\begin{equation*}
\begin{split}
|I_1(\r 1_\O)(x)|&\le\left|\int_{|x-y|<\dl}|x-y|^{1-n}\r(y)1_\O(y) dy\right|+\left|\int_{|x-y|\ge \dl}|x-y|^{1-n}\r(y) 1_\O(y)dy\right|\\
&\le \int_{|x-y|<\dl}|x-y|^{1-n}\r(y)dy+\int_{|x-y|\ge\dl }|x-y|^{1-n}\r(y)1_\O(y)dy\\
&=I+II,
\end{split}
\end{equation*}
where $\dl>0$ is to be determined later. Let $k\in\mathbb{Z}$ and $a_k(x)=\{y: 2^k\dl\le |x-y|<2^{k+1}\dl\},$ then
\begin{equation*}
\begin{split}
I&\le \sum_{k=1}^\infty \int_{a_{-k}(x)}|x-y|^{1-n}\r(y)dy\\
&\le \sum_{k=1}^\infty (2^{-k}\dl)^{1-n}(2^{-k+1}\dl)^n \mathcal{M}_0\r(x)\\
&\le C\dl \mathcal{M}_0\r(x)\quad\textrm{for all}\quad x\in\R^n.
\end{split}
\end{equation*}
Similarly, since $n-2\le \gamma(n)<n-1$ as $\gamma\in[1,n/2]$, there holds
\begin{equation*}
\begin{split}
II&\le \sum_{k=0}^\infty \int_{a_k(x)}|x-y|^{1-n}\r(y) 1_\O(y)dy\\
&\le \sum_{k=0}^\infty (2^k\dl)^{1-n}\mathcal{M}(\r 1_\O)(2^{k+1}\dl, x)\\
&\le C\sum_{k=0}^\infty (2^k\dl)^{1-n}(2^{k+1}\dl)^{\gamma(n)+\alpha}\\
&\le C(n,\gamma,\alpha)\dl^{\gamma(n)-n+1+\alpha}\quad\textrm{for all}\quad x\in \R^n,
\end{split}
\end{equation*}
due to \eqref{300} and \eqref{300a1}.

We choose $\dl=\dl(x)=[\mathcal{M}_0\r(x)]^{1/(\gamma(n)-n+\alpha)}$. Clearly $\dl(x)$ is finite and positive almost everywhere. This choice easily gives
\begin{equation}\label{r1}
|I_1(\r 1_\O)(x)|\le C(n,\gamma,\alpha)[\mathcal{M}_0\r(x)]^{\f{\gamma(n)-n+1+\alpha}{\gamma(n)-n+\alpha}}\quad\textrm{for all}\quad x\in\R^n.
\end{equation}
Hence for $\gamma>1$ the well known $L^\gamma$ inequality for the maximal function $\mathcal{M}_0\r$ yields
\begin{equation}\label{r2}
\begin{split}
\|I_1(\r 1_\O)\|_{L^p(\R^n)}\le C(\gamma, n, \alpha)\|\mathcal{M}_0\r\|^{\f{\gamma(n)-n+1+\alpha}{\gamma(n)-n+\alpha}}_{L^\gamma(\R^n)}\le C(n,\gamma, \alpha)\|\r\|_{L^\gamma(\R^n)}^{\f{\gamma(n)-n+1+\alpha}{\gamma(n)-n+\alpha}}
\end{split}
\end{equation}
with $$p=\gamma\f{\gamma(n)-n+\alpha}{\gamma(n)-n+1+\alpha}>n$$
as required. The inequality \eqref{r2} still holds true for $\gamma=1$ with the last inequality being replaced by \eqref{30b2}.
\end{proof}

\begin{Remark}
Note that for all $\alpha>0$ satisfying \eqref{300}, there holds true
$$\gamma\f{\gamma(n)-n+\alpha}{\gamma(n)-n+1+\alpha}>\gamma\f{\gamma(n)-n}{\gamma(n)-n+1}=n.$$ The latter identity is the reason why the parameter $\gamma(n)$ appears.
\end{Remark}

Given a family of nonnegative densities with the uniform bound \eqref{energy3}, does there exist a sequence all of whose members exhibit prescribed decay requested in Lemma \ref{L52} on the complement of a small uniform set? This question can be solved in an affirmative way through a study of the associated Bessel potentials. The following uniformization theorem concerns the uniform radial decay of density maximal functions for a family of densities.

\begin{Theorem}\label{t3}
Consider a family $\{\r^\i\}$ of nonnegative densities on $\R^n$ with the uniform bound \eqref{energy3}. Fix the small parameter $\alpha>0$. Then there exists a subsequence $\{\r^\i\}$ with the following property: for every $r>0$ there exists a closed set $E_r$ such that
\begin{equation}\label{3001}
E_r\subset\{x\in\R^n: \mathcal{M}\r^\i(s,x)\le s^{\gamma(n)+\alpha}, \quad 0\le s\le 1,\quad \i\le r\}
\end{equation}
and the $(\Gamma(n)+\alpha)-$order Hausdorff premeasure of $E_r^c$ at the level $r$ satisfies
\begin{equation}\label{300a}
\mathcal{H}_r^{\Gamma(n)+\alpha}(E_r^c)\le c(r)
\end{equation}
where $c(r)\rightarrow 0$ as $r\rightarrow 0$.
\end{Theorem}

The smaller the value of $r$ the farther along the sequence it is necessary to go in order to achieve uniformity in decay of the density maximal function. The choice of $\i\le r$ is taken for convenience. 

\begin{proof}[Proof of Theorem \ref{t3}]
For $\alpha\in (0, 1)$ satisfying the inequality \eqref{300},
we define
$$\phi^\i(x)=g_{n-\gamma(n)-\alpha}\star \r^\i (x)=\int_{\R^n}g_{n-\gamma(n)-\alpha}(x-y) \r^\i(y)dy,$$
where $g_{n-\gamma(n)-\alpha}(x)>0$ is the Bessel kernel. Note that there are two positive constants $C_1$ and $C_2$ such that, see (2.6.3) and (2.6.4) in \cite{WP}
\begin{subequations}\label{30a}
\begin{align}
g_{n-\gamma(n)-\alpha}(x)&\le C_1|x|^{-(\gamma(n)+\alpha)} e^{-C_2|x|}\label{30a1}\\
|\nabla g_{n-\gamma(n)-\alpha}|(x)&\le C_1|x|^{-(\gamma(n)+\alpha+1)} e^{-C_2|x|}.\label{30a2}
\end{align}
\end{subequations}
Taking the gradient of $\phi^\i(x)$ gives
\begin{equation}\label{30b3}
\nabla\phi^\i(x)=\int_{\R^n}\nabla g_{n-\gamma(n)-\alpha}(x-y) \r^\i(y)dy.
\end{equation}

\texttt{Step 1: The bound of $\phi^\i$ in $W^{1, n\gamma/(n-\gamma(n-1-\gamma(n)-\alpha))}(B_R)$.}

Due to the bound \eqref{30a2}, one has $|\nabla \phi^\i|\le C I_{n-1-\gamma(n)-\alpha}\r^\i$ and hence
as $\gamma>1$ the classical estimate of Riesz's potential operators yields
\begin{equation*}
\begin{split}
\|\nabla \phi^\i\|_{L^{n\gamma/(n-\gamma(n-1-\gamma(n)-\alpha))}(\R^n)}&\le \|I_{n-1-\gamma(n)-\alpha}\r^\i\|_{L^{n\gamma/(n-\gamma(n-1-\gamma(n)-\alpha))}(\R^n)}\\
&\le C\|\r^\i\|_{L^\gamma(\R^n)},
\end{split}
\end{equation*}
while as $\gamma=1$ Proposition 3.2 in \cite{AD1} and the inequality \eqref{30b2} give
\begin{equation*}
\begin{split}
\|\nabla \phi^\i\|_{L^{n\gamma/(\gamma(n)+1+\alpha)}(B_R)}&\le \|I_{n-1-\gamma(n)-\alpha}\r^\i\|_{L^{n\gamma/(\gamma(n)+1+\alpha)}(B_R)}\\
&\le C\|\mathcal{M}_n\r^\i\|_{L^\infty(B_R)}^{(n-1-\gamma(n)-\alpha)/n}\|\mathcal{M}_0\r^\i\|_{L^1(B_R)}^{(\gamma(n)+1+\alpha)/n}\\
&\le C(R)\|\r^\i\|_{L^1(\R^n)}^{(n-1-\gamma(n)-\alpha)/n}(1+\|\r^\i\ln(1+\r^\i)\|_{L^1(\R^n)})^{(\gamma(n)+1+\alpha)/n}\\
&\le C(R,n)\Big(1+\|\r^\i\|_{L^1(\R^n)}+\|\r^\i\ln(1+\r^\i)\|_{L^1(\R^n)})\Big)
\end{split}
\end{equation*}
as desired. 

A similar procedure further implies that
\begin{equation*}
\begin{split}
\|\phi^\i\|_{L^{n\gamma/(n-\gamma(n-\gamma(n)-\alpha))}(B_R)}
&\le C\|I_{n-\gamma(n)-\alpha}\r^\i\|_{L^{n\gamma/(n-\gamma(n-\gamma(n)-\alpha))}(B_R)}\\
&\le C(R, \gamma, n) \Big(1+\|\r^\i\|_{L^\gamma(\R^n)}+\|\r^\i\ln(1+\r^\i)\|_{L^1(\R^n)})\Big),
\end{split}
\end{equation*}
for $\gamma\ge 1$, and hence a uniform bound of $\|\phi^\i\|_{L^{n\gamma/(n-\gamma(n-\gamma(n)-\alpha))}(B_R)}$ follows.

\texttt{Step 2: The strong convergence of $\phi^\i$ in $W^{1,p}(B_R)$ with $p\in (1, n\gamma/(n-\gamma(n-1-\gamma(n)-\alpha)))$.} In view of \eqref{30b3} and \eqref{30a2}, one has
\begin{equation}\label{30}
\begin{split}
|(\nabla \phi^\i-\nabla\phi)(x)|&
= \left|\int_{\R^n}\nabla g_{n-\gamma(n)-\alpha}(x-y)(\r^\i-\r)(y)dy\right|\\
&\le C(\gamma,n,\alpha)\Big(\int_{B_{\sqrt{\kappa}}(x)}\f{1}{|x-y|^{1+\gamma(n)+\alpha}}|\r^\i-\r|(y)dy\\
&\qquad+\left|\int_{\R^n\setminus B_{\sqrt{\kappa}}(x)}\nabla g_{n-\gamma(n)-\alpha}(x-y)(\r^\i-\r)(y)dy\right|\Big)\\
&=I+II.
\end{split}
\end{equation}
The uniform bound \eqref{energy3} implies that the sequence $\{\r^\i\}$ is equi-integrable in $L^1$ locally (see Theorem 2.10 in \cite{F}); that is for any $\eta>0$ there exists $\kappa>0$ such that
$$\int_{\O}|\r^\i(y)|dy<\eta$$
for any measurable set $\O\subset K$ with $|\O|<\kappa$. 
Hence H\"{o}lder's inequality implies further that for $p>1$
\begin{equation*}
\begin{split}
|I|^p&\le \int_{B_{\kappa^{1/n}}(x)}\f{1}{|x-y|^{p(1+\gamma(n)+\alpha)}}|\r^\i-\r|(y)dy\|\r^\i-\r\|_{L^1(B_{\kappa^{1/n}}(x)}^{p-1}\\
&\le \eta^{p-1}\int_{B_{\kappa^{1/n}}(x)}\f{1}{|x-y|^{p(1+\gamma(n)+\alpha)}} |\r^\i-\r|(y)dy,
\end{split}
\end{equation*}
and hence Fubini's theorem yields
\begin{equation*}
\begin{split}
\int_{B_R}|I|^p dx&\le \eta^{p-1}\int_{(B_R)_{\kappa^{1/n}}}\int_{B_R}\f{1}{|x-y|^{p(1+\gamma(n)+\alpha)}} dx|\r^\i-\r|(y)dy\\
&\le \eta^{p-1}
\end{split}
\end{equation*}
if $p\in (1, n/(\gamma(n)+1+\alpha))$, where $\O_r=\{x\in\R^n: \textrm{dist}(x,\O)\le r\}$.

For II, since the function $\nabla g_{n-\gamma(n)-\alpha}(x)$ is continuous in $\{|x|\ge \kappa^{1/n}\}$, one has
$$\textrm{II}\rightarrow 0\quad\textrm{ for all }\quad x\in B_R \quad\textrm{as}\quad \i\rightarrow 0$$
due to the weak convergence of $\r^\i$ in $L^1$. Then Egorov's theorem and the uniform bound in Step 1 imply that
$$\int_{B_R}|II|^p dx\rightarrow 0\quad\textrm{as}\quad\i\rightarrow 0$$
for any fixed $\kappa$. Therefore letting $\i$ tends to zero in \eqref{30} yields
\begin{equation}\label{30a}
\limsup_{\i\rightarrow 0}\int_{B_R}|\nabla \phi^\i-\nabla\phi|^p(x) dx\le \eta^{p-1}.
\end{equation}
Since $\eta$ is arbitrary, one deduces from \eqref{30a} that
\begin{equation*}
\lim_{\i\rightarrow 0}\int_{B_R}|\nabla \phi^\i-\nabla\phi|^p(x) dx=0.
\end{equation*}
This strong convergence, combining the interpolation inequality and the bound in Step 1, yields the desired strong convergence of $\nabla\phi^\i$ in $L^{p}(B_R)$ with $p\in (1, n\gamma/(n-\gamma(n-1-\gamma(n)-\alpha)))$.
Moreover a similar argument works also for $\phi^\i$ itself and hence $\phi^\i\rightarrow \phi$ in $W^{1,p}(B_R)$ as $\i\rightarrow 0$.

\texttt{Step 3: Construction of the exceptional set $E_r^c$.}
We extend $\phi^\i$ which is only bounded in $W^{1,p}(B_R)$ to a function, still denoted by $\phi^\i$, in $W^{1,p}(\R^n)$. Moreover we have $\phi^\i\rightarrow \phi$ in $W^{1,p}(\R^n)$ with $p\in (1,n\gamma/(n-\gamma(n-1-\gamma(n)-\alpha)))$ for some function $\phi\in W^{1,p}(\R^n)$. Since $n\gamma/(n-\gamma(n-1-\gamma(n)-\alpha))<n$ and\footnote{These inequalities are verified in Appendix.}
\begin{equation}\label{in}
1<n-\Gamma(n)
\begin{cases}
\le\f{n\gamma}{\gamma(n)+1}<\f{n\gamma}{n-\gamma(n-1-\gamma(n))}\quad\textrm{as}\quad \gamma\in(1,n/2]\\
=\f{n\gamma}{\gamma(n)+1}=\f{n\gamma}{n-\gamma(n-1-\gamma(n))}\quad\textrm{as}\quad \gamma=1,
\end{cases}
\end{equation}
set 
\begin{equation}\label{in1}
p=
\begin{cases}
n-\Gamma(n)\quad\textrm{as}\quad \gamma>1\\
\f{n\gamma}{\gamma(n)+1+2\alpha}\quad\textrm{as}\quad \gamma=1
\end{cases}
\end{equation}
and choose a small $\alpha>0$ such that
$$1<p<\f{n\gamma}{n-\gamma(n-1-\gamma(n)-\alpha)}<n\quad\textrm{as}\quad \gamma\in[1,n/2].$$
The strong convergence, combining Lemma 4.2 in \cite{DM} (see also Theorem 12.4 in \cite{BM} and Theorem 7 in \cite{LE}), implies that there exists a subsequence $\phi^\i$ such that $\phi^\i\rightarrow \phi$ pointwise on $\O_1$ and $\mathcal{H}^{\Gamma(n)+\alpha}(\O_1^c)=0$. Since the limit potential $\phi$ belongs to $W^{1,p}(\R^n)$, applying Theorem 5.1 in \cite{DM} (see also \cite{FW,LG}), $\phi$ has a pointwise bound $C(r)$ on $\O_2$ and $$\mathcal{H}_r^{\Gamma(n)+\alpha}(\O_2^c)\le cr^{n-p}$$ for a universal positive constant $c$. A union of these two exceptional sets $E_r^c=\O_1^c\cup\O_2^c$ provides a uniform pointwise bound on a sequence of potentials $\phi^\i$ on $E_r$ and $E_r^c$ satisfies \eqref{300a}. Thus for $x\in E_r$ there exists a constant $C(r)$ such that
$$\phi^\i(x)\le C(r)\quad\textrm{for}\quad \i\le r.$$ Since the kernels of Bessel potentials and Riesz potentials are comparable near the origin (see Section 2.6 in \cite{WP}), the uniform estimate above further yields that for $s\in (0,1)$
\begin{equation*}
\begin{split}
s^{-\gamma(n)-\alpha} \mathcal{M}\r^\i(s,x)&\le \int_{B_s(x)}|x-y|^{-\gamma(n)-\alpha} \r^\i(y)dy\\
&\le C(\gamma,n)\int_{B_s(x)}g_{n-\gamma(n)-\alpha}(x-y)\r^\i(y)dy\le C(\gamma, n)\phi^\i(x)\\
&\le C(r,\gamma, n).
\end{split}
\end{equation*}
as required in \eqref{3001}.
\end{proof}

As a consequence, Theorem \ref{t3}, combining Proposition \ref{p6}, implies that the Hausdorff dimension of the concentration for the reduced defect measure $\theta$ is $\Gamma(n)$.
\begin{Theorem}\label{mt}
The reduced defect measure $\theta$ concentrates inside a set $E$ with Hausdorff dimension $\Gamma(n)$, i.e.
$$\mathcal{H}^{\Gamma(n)+\alpha}(E)=0\quad \textrm{for all}\quad \alpha>0.$$
\end{Theorem}
\begin{proof}
This is a direct consequence of Proposition \ref{p6} and Theorem \ref{t3}, by noting that $\gamma(n)\ge n-2$ as $\gamma\in [1, n/2]$.
\end{proof}

\begin{Remark}
Theorem \ref{mt} implies that the Hausdorff dimension of the concentration set has the dimension zero if $n=2$, and hence it is actually a counting measure; that is there exists a counting measure such that
$$\theta=\sum_{j=1}^\infty \theta_j\dl_{x_j}.$$ This provides an alternative way to explain the result of Lions in the time-discretized case (see Lemma 6.1 in \cite{PL}), and keeping in mind, there always holds $\theta(E)\le \nu(E)$. Moreover even if concentrations of the kinetic energy occur in dimensions two, the convective term is insensitive to that concentration in the steady or time-discretizied case (see \cite{FSW, PL}); that is the weak limit $(\r,\u)$ of $(\r^\i,\u^\i)$ is a weak solution of the 2D isothermal compressible Navier-Stokes equations even though $\r^\i\u^\i\otimes\u^\i\rightarrow \r\u\otimes\u$ might not be valid generally.
\end{Remark}

\bigskip\bigskip


\section{Space-Time Defects}

This section aims at the size of space-time $L^2$ defects of $\u^\i$ with the weight $\r^\i$ for the approximate solution sequence $(\r^\i,\u^\i)$ of $n-$dimensional isentropic compressible Navier-Stokes equations \eqref{e1}. \textit{A priori} estimate \eqref{energy1} yields the following uniform bound for the sequence
\begin{equation}\label{40}
\sup_{t\ge 0} E(\r^\i(t), \u^\i(t))+\int_0^\infty\Big(\mu\|\nabla\u^\i\|_{L^2(\R^n)}^2(t)+\xi\|\Dv\u^\i\|_{L^2(\R^n)}(t)\Big)dt\le C
\end{equation} 
with $C$ a fixed constant.

First of all we set up the H\"{o}lder estimate of $I_1\r^\i$ for the temporal variable. 
\begin{Lemma}\label{L44}
Let $\Phi(x)\in W^{1,\infty}(\R^n)$ and $\r^\i$ be an approximate solution sequence of \eqref{e1}, then for any $p$ with $1\le p<n/(n-1)$, there is a number $\theta_0\in (0,1)$ such that
$$\|I_1[\Phi\r^\i(t_1)]-I_1[\Phi\r^\i(t_2)]\|_{L^p(\R^n)}\le C|t_1-t_2|^{\theta_0}$$
for a universal constant $C>0$.
\end{Lemma}
\begin{proof}
From the first equation of \eqref{e1}, $I_1\r^\i$ satisfies the equation
$$\partial_tI_1[\Phi\r^\i]=-I_1[\Phi\Dv(\r^\i\u^\i)]$$ and hence the uniform bound $\|\r^\i\u^\i\|_{L^\infty(0,T; L^1(\R^n))}$ implies that the sequence $I_1[\Phi\r^\i]$ is a uniformly Lipschitz function with values in $W^{-L,2}_{loc}(\R^n)\subset W^{-L,p}_{loc}(\R^n)$ with $p\in (1, 2)$ for some large integer $L>0$. This uniform bound guarantees that for $p\in(1, 2)$
\begin{equation}\label{30f}
\|I_1[\Phi\r^\i(t_1)]-I_1[\Phi\r^\i(t_2)]\|_{-L,p}\le C|t_1-t_2|
\end{equation}
for some fixed constant $C$. 

Moreover the function $I_1[\Phi\r^\i]$ satisfies the elliptic equation
\begin{equation}\label{41aaaaa}
(-\D)^{\f12} I_1[\Phi\r]=\Phi\r.
\end{equation}
Due to Sobolev's lemma, $W^{s,p'}\mapsto C(\R^n)$ for $s>N/p'$, so by duality $L^1\mapsto W^{-s,p}$ for $s>n((p-1)/p)$ with $p>1$, and therefore the uniform mass bound $\|\r^\i\|_{L^1(\R^n)}$ yields 
$$\|\r^\i\|_{-s,p}\le C\|\r^\i\|_{L^1}\le C\quad \textrm{for}\quad s>n((p-1)/p)$$
and $p>1$. Solutions of the elliptic equation \eqref{41aaaaa} gain one derivative, therefore
$$\|I_1[\Phi\r^\i]\|_{1-s,p}\le C\|\r^\i\|_{L^1}\le C$$ for $s>n(p-1)/p$. Thus for $1<p<n/(n-1)$, we have for some $s_1=1-s>0$
\begin{equation}\label{30e}
\|I_1[\Phi\r^\i(t)]\|_{s_1,p}\le C\quad\textrm{for all}\quad 0\le t\le T.
\end{equation}
From the interpolation inequality, there is a number $\theta_0$ with $0<\theta_0<1$ such that
$$\|I_1[\Phi\r^\i(t_1)]-I_1[\Phi\r^\i(t_2)]\|_{L^p}\le C\|I_1[\Phi\r^\i(t_1)]-I_1[\Phi\r^\i(t_2)]\|_{s_1,p}^{1-\theta_0}\|I_1[\Phi\r^\i(t_1)]-I_1[\Phi\r^\i(t_2)]\|_{-L,p}^{\theta_0}.$$
With the Lipschitz bound \eqref{30f} and the uniform bound \eqref{30e}, the H\"{o}lder estimate follows from the interpolation inequality.
\end{proof}

Given the sequence of nonnegative temporal densities, we utilize the associated Bessel potentials
$$\phi(x,t)=\int_{\R^n}g_{n-\gamma(n)-\alpha}(x-y)\r(y,t)dy\quad \textrm{with}\quad\alpha\quad\textrm{satisfying \eqref{300}}.$$
The following proposition generalizes Theorem \ref{t3} to the time-dependent case.
\begin{Proposition}\label{p44}
Fix the small parameter $\alpha>0$. Given a family of nonnegative densities satisfying \eqref{40} there exists a subsequence $\r^\i$ with the following property: for each $t$ there exists a collection of closed subsets $F_r(t)\subset\R^n$, $r>0$, such that
\begin{equation}\label{41}
\mathcal{H}^{\Gamma(n)+\alpha}_r(F_r^c(t))\le c(r)
\end{equation}
and 
\begin{equation}\label{41a}
\phi^\i(x,t)\le C(r)\quad\textrm{if}\quad x\in F_r(t)\quad\textrm{and}\quad \i<r. 
\end{equation}
The constant $C(r)$ in \eqref{41a} is independent of $t$, and the time independent constant $c(r)$ tends to zero as $r\rightarrow 0$.
\end{Proposition}

The pointwise bound \eqref{41a} implies that the associated time-dependent density maximal function
$$\mathcal{M}\r(s, x,t)=\int_{B(x,s)}\r(y,t)dy$$
exhibit temporally uniform $(\gamma(n)+\alpha)-$ order decay if $x$ lies in the complement of the exceptional set $F_r(t)$; that is 
$$F_r(t)\subset\{x\in\R^n: \mathcal{M}\r^\i(s,x,t)\le c(\alpha)s^{\gamma(n)+\alpha}\}.$$

\begin{proof}
We first verify the temporal Lipschitz continuity of $\phi^\i$ in some negative space, $W^{-L,p}$ through a similar procedure in Lemma \ref{L44}. Indeed the first equation in \eqref{e1} and the uniform bound $\|\r^\i\u^\i\|_{L^\infty(0,\infty; L^1(\R^n))}$ imply that the density $\r^\i(x,t)$ satisfies the temporal Lipschitz estimate
\begin{equation}\label{41aa}
\|\r^\i(t_1)-\r^\i(t_2)\|_{W^{-L,p}}\le c|t_1-t_2|
\end{equation}
for $p\in (1,\infty)$ and some sufficiently large $L>0$.
It is well known that the Bessel potential $g_{n-\gamma(n)-\alpha}$ maps $W^{-L,p}$ boundedly into $W^{-L+n-\gamma(n)-\alpha-\i, p}$ for any $\i>0$ and $p\in (1,\infty)$; this fact, combining the estimate \eqref{41aa}, yields the required negative-norm temporal Lipschitz estimate of $\phi^\i$. Therefore the Lions-Aubin lemma and a similar procedure as Theorem \ref{t3} yield that
\begin{equation}\label{41aa1}
\phi^\i(x,t)\rightarrow \phi(x,t)\quad\textrm{in}\quad C([0,T]; W^{1,p}(B_R))
\end{equation}
with $p\in(1,n\gamma/(n-\gamma(n-1-\gamma(n)-\alpha)))$ as $\i\rightarrow 0$. With \eqref{41aa1} in hand, the exceptional set $F_r^c(t)$ with \eqref{41} and \eqref{41a} is constructed in the spirit of Theorem \ref{t3}.
\end{proof}

Next we go beyond from the pointwise exceptional sets with respect to $t$ in Proposition \ref{p44}  to the construction of a space-time exceptional set. Indeed with the above facts in hand, we formulate a general theorem which states that except a space-time set with Hausdorff dimensions less than $\Gamma(n)+1$, the Riesz potential of the density has higher integrability in $L^s(\R^n\times\R^+)$ with $s>n$. The strategy of the proof combines temporal compactness of Bessel potentials $\phi^\i$ and H\"{o}lder estimates in time together with the argument for elliptic sequences at fixed time presented in the previous section. 
\begin{Theorem}\label{4m}
Fix positive numbers $\theta_0, \alpha, \beta, p$ and $q$ such that $1<q<n<p$. Suppose $\{\r^\i,\u^\i\}_{\i\ge 0}$ satisfies the uniform estimate \eqref{40} and the property: for every $t$ in $[0,T]$ there exists a family of closed sets $F_r(t)\subset\R^n$, $r>0$, such that
\begin{equation}\label{440}
\mathcal{H}^{\Gamma(n)+\alpha}_r(F^c_r(t))\le c(r)
\end{equation}
and all members $(\r^\i,\u^\i)$ of the sequence satisfy
\begin{subequations}\label{44}
\begin{align}
&\|I_1[\Phi\r^\i](\cdot, t_1)-I_1[\Phi\r^\i](\cdot, t_2)\|_{L^q}\le c |t_1-t_2|^{\theta_0}\label{44a}\\
&\|I_1[\r^\i1_{F_r(t)}]\|_{L^p(\R^n)}\le c,\label{44b}
\end{align}
\end{subequations}
where $c$ is a fixed constant, the time independent constant $c(r)$ tends to zero as $r\rightarrow 0$, and $\Phi(x)\in W^{1,\infty}(\R^n)$.
Then there exists a family of closed sets $G_r\subset \R^n\times[0,T]$ with $r>0$ such that
$$\mathcal{H}_r^{\Gamma(n)+\alpha,1+\beta}(G_r^c)\le c_1(r)$$ and all members $\r^\i$ of the sequence satisfy
$$\|I_1[\r^\i 1_{G_r}] (x,t)\|_{L^s(\R^n\times [0,T])}\le c_1,$$ where $\lim_{r\rightarrow 0}c_1(r)=0$ and the constant $c_1$ depends on $c$ and $s\in (n,p)$.
\end{Theorem}

Proposition \ref{p44} ensures the uniform bound \eqref{440} of Hausdorff $(\Gamma(n)+\alpha)-$ premeasure of exceptional sets in $t$. The conditions \eqref{44a}-\eqref{44b} are guaranteed by Lemma \ref{L44} and Lemma \ref{L52} respectively. As a result, Theorem \ref{4m} tells that on the set where $\r$ decays faster (or concentrates slowly), the space-time integrability of $I_1\r$ could be improved in $L^s$ with $s>n$.
 
\begin{proof}
The strategy is to estimate the restricted distribution functions of the form
$$\sigma(\lambda)=m_{n+1}\{(x,t)\in \R^n\cap [0,T]: I_1[\r^\i 1_{G_r(t)}](x,t)\ge \lambda\}$$
where $m_d(\cdot)$ denotes d-dimensional Lebesgue measure and $G_r(t)$ denotes the cross section of $G_r$ at time $t$.

We shall show that $\sigma\le c\lambda^{-p}$ by estimating a Riemann sum over a finite partition by $\{t_j: 1\le j\le m 2^N\}$ with equal length where $m$ and $N$ are integers chosen in an appropriate fashion later:
\begin{equation}\label{4401}
\sigma(\lambda)=\int_0^T f(t,\lambda)dt\le \sum_j\sup\{|f(t,\lambda)|:t\in I_j\}\D t
\end{equation}
where
$$f(t,\lambda)=m_n\{x\in \R^n:I_1[\r^\i1_{G_r(t)}](x,t)\ge \lambda\}\quad\textrm{and}\quad I_j=\{t: |t-t_j|\le \D t\}.$$
Precisely, the mesh consists of $m2^N$ points $t_j$ that partition $[0,T]$ into intervals of equal length $\D t$.

We construct the exceptional set  $G_r$ as follow. Fix $r$ and select a point $t_j$ of the mesh. Consider the spatial exceptional set $F_j^c=F_r^c(t_j)$. Notice that the Hausdorff $\alpha$-order premeasure of $F_r^c$ is bounded uniformly with respect to $j$ and $r$ according to \eqref{440}. Without loss of generality we may assume that $F_j^c$ is a countable union of balls $B_{j,k}\subset\R^n$ with radii $r_{j,k}$ satisfying
$$\sum r_{j,k}^{\Gamma(n)+\alpha}\le c(r)\quad\textrm{and}\quad r_{j,k}\le r.$$
Associated with each ball we define a cylinder with height $h_j$:
$$C_{j,k}=B_{j,k}\times \{t:|t-t_j|< h_j/2\}.$$
We define the space-time exceptional set $G_r^c$ to be a union of cylinders, namely
\begin{equation}\label{G}
G_r^c=\bigcup_{j,k}C_{j,k}
\end{equation}
and require that the heights satisfy the following inequalities
$$\sum h_j^{1+\beta}\le c\quad \textrm{and}\quad h_j\le r$$
where $c$ is a fixed constant. We observe that the cylinder premeasure of $G_r^c$ is bounded uniformly with respect to $r$, which further implies that the cylinder premeasure satisfies the following inequality
$$\mathcal{H}^{\Gamma(n)+\alpha,1+\beta}(G_r^c)\le cc(r)\coloneqq c_1(r).$$

In order to obtain the desired upper bound on $\sigma$ we shall first show that 
\begin{equation}\label{4403}
\sigma(\lambda)\le cT\lambda^{-p}+cT(\D t)^{\theta_0 q}\lambda^{-q}.
\end{equation}
To this end, we shall show that
\begin{equation}\label{4402}
f(t,\lambda)\le A_j+B_j\quad\textrm{if}\quad t\in I_j,
\end{equation}
where
$$A_j=c\lambda^{-p}\quad\textrm{and}\quad B_j=c|t-t_j|^{\theta_0 q}\lambda^{-q}.$$
Substitution of the pointwise bounds into the Riemann sum \eqref{4401} yields the desired \eqref{4403}.

In order to prove \eqref{4402}, we first observe that if $t$ lies in $I_j$ then the $t-$section of $G_r^c$ contains $F_j^c$. Then the nonnegative of $\r^\i$ implies 
$\r^\i 1_{G_r(t)}\le \r^\i 1_{F_j}$, and hence
$$f(t,\lambda)\le m_n\{x\in \R^n: I_1[\r^\i 1_{F_j}]\ge \lambda\}\quad\textrm{if}\quad t\in I_j.$$
The triangle inequality implies that
\begin{equation*}
\begin{split}
f(t,\lambda)&\le m_n\{x\in \R^n: I_1[\r^\i(\cdot,t_j)1_{F_j}]>\lambda/2\}\\
&\quad+m_n\{x\in \R^n: |I_1[\r^\i(x,t)1_{F_j}]-I_1[\r^\i(x,t_j)1_{F_j}]|>\lambda/2\}
\end{split}
\end{equation*}
and therefore \eqref{4402} holds with
$$A_j=m_n\{x\in \R^n: I_1[\r^\i(x,t_j)1_{F_j}]>\lambda/2\}$$
and
$$B_j=m_n\{x\in \R^n: |I_1[\r^\i(x,t)1_{F_j}]-I_1[\r^\i(x,t_j)1_{F_j}]|>\lambda/2\}\quad\textrm{if}\quad t\in I_j.$$
The restricted $L^p$ bound \eqref{44b} yields the desired bound on the spatial distribution function of the restriction of $I_1[\r^\i(\cdot, t_j) 1_{F_j}]$ to $F_j$ and hence on $A_j$, while the temporal H\"{o}lder estimate \eqref{44a} yields the desired bound in $B_j$ since $1_{F_j}=\prod_{k=1}^\infty 1_{B_{j,k}^c}\in W^{1,\infty}(\R^n)$ due to the construction of $F_j$. This verifies the estimate \eqref{4402}.

Finally, the integer $N$ is chosen so that the factor $(\D t)^{\alpha q}$ in \eqref{4403} compensates for the large term $\lambda^{-q}$ on the right side of \eqref{4403}. We rewrite \eqref{4403} in the form
$$\sigma\le cT\lambda^{-p}+cT\lambda^{-p}\pi$$
where
$$\pi=(\D t)^{\theta_0 q}\lambda^{p-q}=(2^{-N-1}T)^{\theta_0 q}\lambda ^{p-q}.$$
We fix $\lambda$ and choose $N$ sufficiently large to guarantee that $\pi$ is dominated by a fix constant.

It is convenient to regard the set of points $t_j$ as a union of a sequence of finite meshes $M_k$ defined inductively as follows. Assume with loss of generality that $T/r$ is an integer and let $m=T/r$. Take $M_1$ to consist of $t_j=jr$ where $1\le j\le m$. Given $M_k$ define $M_{k+1}=M_k\cup C_k$ where $C_k$ denotes the collection of midpoints of the subintervals of $(0,T)$ determined by $M_k$. Notice that $M_{k+1}$ contains twice as many points as $M_k$. Thus $M_k$ consists of $m2^k$ points $t_j$, $1\le j\le m2^j$, which partition $(0,T)$ onto intervals of equal length and may be regarded as a disjoint union, $M_k=\cup_{i=0}^{k-1}C_i$ by taking $C_0=M_1$.

If $t_j$ lies in $C_i$, define the corresponding $h_j$ to be $r2^{-i}$, that is the distance between consecutive points of $C_i$. Thus the sum of the heights associated with points of $C_i$ equals $T$ and 
\begin{equation*}
\begin{split}
\sum h_j^{1+\beta}&=\sum_{i}\sum_{t_j\in C_i}h_j^{1+\beta}\le \sum_{i}(r2^{-i})^\beta\sum_{t_j\in C_i}h_j\\
&\le T\sum_{i}(r2^{-i})^\beta\le c,
\end{split}
\end{equation*}
provided that $\beta>0$.
\end{proof}

\bigskip\bigskip


\section{Proof of Theorem \ref{mt2}}

In the previous two sections, we constructed an exceptional space-time set $G^c_r$ such that on $G_r$ the Riesz potential of the density has higher integrability such as $L^s$ for some $s>n$ and the exceptional space-time set is constructed as a subset of the complement of sets where the radial distribution functions of density, $\mathcal{M}\r(s,x,t)=\|\r(\cdot, t)\|_{L^1(B_s(x))}$, decay faster. In this section, based on this higher integrability of the Riesz potential of the density, we further improve the integrability of the density itself in $L^{\gamma+\lambda}$ on the same space-time set $G_r$ for some $\lambda>0$. This improvement, combined with the classical convexity method (\cite{PL, F, NP}), yields the strong convergence of $\r^\i$, and hence the part of Theorem \ref{mt2} on the Hausdorff dimension of the reduced defect measure follows. 

\subsection{Improved density estimates}

On the set $G_r$, the higher integrability of the Riesz potential of the density further improves the integrability of the density itself due to the moment equation. Precisely, there holds true
\begin{Proposition}\label{p3}
Let $G_r$ be the set constructed in Theorem \ref{4m}. There is a positive number $\lambda>0$ such that
$$\|\r^{\gamma+\lambda}\|_{L^1(G_r\cap (B_R\times(0,T)))}\le C.$$
Here the constant $\lambda$ depends on the parameter $n, \gamma$ in the way that
\begin{equation}\label{50a}
0<\lambda\le \min\left\{\f{\gamma}{n+1},\f{(\tau-2)\gamma}{2\tau}\right\},
\end{equation}
where $\tau$ is given by the identity \eqref{52a}.
\end{Proposition}

This and next two subsections are devoted to the proof of Proposition \ref{p3}. First of all we recall the renormalised continuity equation for $b(y)=y^\lambda$ (see \cite{PL, F, NP})
\begin{equation}\label{50}
\begin{split}
\partial_t\r^\lambda+\Dv(\u \r^\lambda)+(\lambda-1) \r^\lambda\Dv\u=0
\end{split}
\end{equation}
in the sense of distributions.

Due to the construction of $G_r$ in \eqref{G}, $G_r^c$ is a countable union of open cylinders. Hence $G_r\cap \ov{B_R\times(0,T/2)}$ and $G_r\cap [B_{2R}\times (0,T)]^c$ are disjoint closed sets. Using Urysohn's lemma as in \eqref{u}, there exists a Lipschitz continuous function $0\le \psi\le 1$ such that $\psi$ is equal to one in $G_r\cap \ov{B_R\times (0,T/2)}$ and vanishes in $G_r\cap [B_{2R}\times(0,T)]^c$.
Denote
$$\mathcal{A}_j(f)=\partial_j(-\D)^{-1}f.$$
The operator $\mathcal{A}$ has the property (see \cite{F, NP}): $$\|\mathcal{A}_i(f)\|_{W^{1,p}(\R^n)}\le C\|f\|_{L^p(\R^n)}.$$
Multiplying the momentum equation by $\psi\mathcal{A}_j(\r^\lambda)$, we obtain
\begin{equation}\label{pes}
\begin{split}
\int_0^T\int_{\R^n}\psi \r^{\gamma+\lambda}dxdt =\sum_{i=1}^4I_i +\sum_{j=1}^6J_j,
\end{split}
\end{equation}
where 
\begin{equation*}
\begin{split}
I_1&=\int_0^T\int_{\R^n}\Big[\r\u_i\u_j \partial_j\mathcal{A}_i(\r^\lambda)-\r\u_i\partial_j\mathcal{A}_i[\u_j\r^{\lambda}] \Big]\psi dxdt\\
I_2&=-\mu\int_0^T\int_{\R^n}\partial_j\u^i \partial_j\mathcal{A}_i(\r^\lambda) \psi dxdt\\
I_3&=-\xi\int_0^T\int_{\R^n}\Dv\u \partial_i\mathcal{A}_i(\r^\lambda)\psi dxdt\\
I_4&=-(\lambda-1)\int_0^T\int_{\R^n}\r\u_i\mathcal{A}_i\left[\r^\lambda\Dv\u\right] \psi dxdt,\\
&
\end{split}
\end{equation*}
and
\begin{equation*}
\begin{split}
J_1&=\int_0^T\int_{\R^n}\partial_j\psi \mathcal{A}_j(\r^\lambda)\r^\gamma dxdt\\
J_2&=\int_0^T\int_{\R^n}\r\u_i\u_j \mathcal{A}_j(\r^\lambda)\partial_i\psi dxdt\\
J_3&=\int_0^T\int_{\R^n}\partial_t\psi \r\u_j\mathcal{A}_j(\r^\lambda)dxdt\\
J_4&=-\mu\int_0^T\int_{\R^n}\partial_j\u_i \mathcal{A}_i(\r^\lambda) \partial_j\psi dxdt\\
J_5&=-\xi\int_0^T\int_{\R^n}\Dv\u \mathcal{A}_i(\r^\lambda)\partial_i\psi dxdt.
\end{split}
\end{equation*}

The most difficult terms in the right hand side of \eqref{pes} are $I_1$ and $I_4$. The strategy to handle terms in the right hand side of \eqref{pes} goes as follows.
\begin{itemize}
\item For the first term $I_1$ in the right, the concentration cancellation occurs. More precisely,  after integration by parts, the first term can be rewritten as
\begin{equation*}
\begin{split}
I_1&=\int_0^T\int_{\R^n}\Big[\r\u_i\u_j \partial_j\mathcal{A}_i(\r^\lambda)-\r\u_i\partial_j\mathcal{A}_i[\u_j\r^{\lambda}]\Big]\psi dxdt\\
&=\int_0^T\int_{\R^n}\u_j\Big[\psi\r\u_i\partial_j\mathcal{A}_i(\r^\lambda)-\partial_j\mathcal{A}_i(\psi\r\u_i)\r^{\lambda}\Big] dxdt
\end{split}
\end{equation*}
and the factor of the integrand above has the following identity
\begin{subequations}\label{51}
\begin{align}
&\psi\r\u_i\partial_j\mathcal{A}_i(\r^\lambda)-\partial_j\mathcal{A}_i(\psi\r\u_i) \r^\lambda\nonumber\\
&\quad=\psi\r\u_i\partial_i\mathcal{A}_j(\r^\lambda)-\nabla\mathcal{A}_i(\psi\r\u_i)\cdot e_j\r^\lambda\nonumber\\
&\quad=\Big(\psi\r\u+\nabla(-\D)^{-1}\Dv(\psi\r\u)\Big)\cdot\nabla(-\D)^{-1}\Dv(e_j\r^\lambda)\label{51a}\\
&\qquad-\nabla(-\D)^{-1}\Dv(\psi\r\u)\cdot \Big(e_j\r^\lambda+\nabla(-\D)^{-1}\Dv(e_j\r^\lambda)\Big),\label{51b}
\end{align}
\end{subequations}
where $\{e_j\}_{j=1}^n$ is the classical basis of $\R^n$.
The advantage of the term \eqref{51a} and the term \eqref{51b} is that both of them satisfy the classical Div-Curl structure.
\item For the fourth term $I_4$, an integration by parts yields
$$I_4=(\lambda-1)\int_0^T\int_{\R^n}\mathcal{A}_i[\psi\r\u_i]\r^\lambda \Dv\u dxdt,$$ and an improved estimate for Riesz potentials of the momentum, $\mathcal{A}_i[\psi\r\u_i]$, will be established in Lemma \ref{L51}.
\item All other terms in \eqref{pes} can be dealt straightforwardly by \textit{a priori} estimates \eqref{energy1}.
\end{itemize}

To start with, the estimates for $J_j (1\le j\le 5)$ are demonstrated as follows. For $J_1$ and $J_2$, one has
\begin{equation*}
\begin{split}
|J_1|+|J_2|&\le CT\Big(\|\r|\u|^2\|_{L^\infty(0,T; L^1(\R^n))}+\|\r^\gamma\|_{L^\infty(0,T; L^1(\R^n))}\Big)\|\mathcal{A}_j(\r^\lambda)\|_{L^\infty(B_{2R}\times(0,T))}\\
&\le CT\Big(\|\r|\u|^2\|_{L^\infty(0,T; L^1(\R^n))}+\|\r^\gamma\|_{L^\infty(0,T; L^1(\R^n))}\Big)\|\mathcal{A}_j(\r^\lambda)\|_{L^\infty(0,T; W^{1, n+1}(B_{2R}))}\\
&\le CT\Big(\|\r|\u|^2\|_{L^\infty(0,T; L^1(\R^n))}+\|\r^\gamma\|_{L^\infty(0,T; L^1(\R^n))}\Big)\|\r\|_{L^\infty(0,T; L^\gamma(\R^n))}<\infty
\end{split}
\end{equation*}
since $\lambda(n+1)\le \gamma$. The term $J_3$ can be estimated similarly as $J_1$ since $2|\r\u|\le \r+\r|\u|^2$. Moreover, terms $J_4$ and $J_5$ have the bound
\begin{equation*}
\begin{split}
|J_4|+|J_5|&\le C\|\nabla\u\|_{L^2(\R^2\times(0,T))}\|\mathcal{A}_i(\r^\lambda)\|_{L^2(B_{2R}\times(0,T))}\\
&\le C\sqrt{T}\|\nabla\u\|_{L^2(\R^2\times(0,T))}\|\mathcal{A}_i(\r^\lambda)\|_{L^\infty(0,T; L^2(B_{2R}))}\\
&\le C\sqrt{T}\|\nabla\u\|_{L^2(\R^2\times(0,T))}\|\r\|_{L^\infty(0,T; L^\gamma(\R^n))}<\infty
\end{split}
\end{equation*}
since $2\lambda<(n+1)\lambda\le \gamma$.

The estimations for $I_2$ and $I_3$ are also quite straightforward. Indeed,
for $I_2$, due to the fact $\|\nabla\mathcal{A}_i(f)\|_{L^2(\R^n)}\le C\|f\|_{L^2(\R^n)}$ and $2\lambda\le \gamma$, one has
\begin{equation*}
\begin{split}
|I_2|+|I_3|&\le C\|\nabla\u\|_{L^2((0,T)\times\R^n)}\|\r^\lambda\|_{L^2(B_{2R}\times (0,T))}\\
&\le C\sqrt{T}\|\nabla\u\|_{L^2((0,T)\times\R^n)}\|\r\|_{L^\infty(0,T;L^\gamma(\R^n))}<\infty.
\end{split}
\end{equation*}

In order to handle the terms $I_1$ and $I_4$, we need to introduce 
the function $\F=(\F_1,\F_2)$ where the functions $\F_1$ and $\F_2$ stand for the Hodge decomposition of $\psi\r\u$ for the function $0\le \psi \le1$ in \eqref{pes}; that is 
$$\nabla \F_1+\Cu\F_2=\psi\r\u \quad\textrm{in}\quad \R^n.$$
It is easily seen that
$$\F_1=\Dv(-\D)^{-1}(\psi\r\u)\quad\textrm{and}\quad \F_2=-\Cu(-\D)^{-1}(\psi\r\u).$$
The following auxiliary lemma gives a $L^r$ estimate with $r>2$ for the vector-valued function $\F$ in terms of the energy $E(\r(t),\u(t))$ and the Riesz potential of the density.

\begin{Lemma}\label{L51}
For $\dl\in(0,1-n/s)$, there is a constant $C(\dl)$ such that
\begin{equation*}
\begin{split}
\|\F\|_{L^{2s}(0,T; L^\tau(\R^n))}&\le C(\dl)\Big(\|\r|\u|^2\|_{L^\infty(0,T; L^1(\R^n))}\left[1+\| I_1(\r 1_{G_r})\|_{L^s(\R^n\times[0,T])}\right]\\
&\qquad+\|\r^\gamma\|_{L^\infty(0,T; L^1(\R^n))}\Big),
\end{split}
\end{equation*}
where
\begin{equation}\label{52a}
\f{1}{\tau}=\f{n-1+\dl}{2n}+\f{1}{2s}\quad\textrm{with}\quad s>n\quad \textrm{and}\quad \tau>2.
\end{equation}
\end{Lemma}
The parameter $\dl$ is chosen in $(0, 1-n/s)$ to make sure $\tau>2$.
\begin{proof}
It is easy to see that
$$\F_1(x)=\f{1}{n\varpi_n}\int_{\R^n}\f{x-y}{|x-y|^{n}}\cdot (\psi\r\u)(y)dy\quad\textrm{and}\quad \F_2(x)=\f{1}{n\varpi_n}\int_{\R^n}\f{x-y}{|x-y|^{n}}\times(\psi\r\u)(y)dy.$$

\texttt{Case 1: $|x|>4$.} On one hand, one notices that
$$|\F(x)|\le C\int_{\R^n}\f{1}{|x-y|^{n-1}}\psi\r|\u|(y)dy.$$
As $\gamma>1$, the energy estimate \eqref{energy1} gives $\r\u\in L^\infty(0,T; L^{\f{2\gamma}{\gamma+1}}(\R^n))$, and hence $$I_{1}(\r\u)\in L^\infty(0,T; L^{\alpha_1}(\R^n))\quad\textrm{with}\quad
\f{1}{\alpha_1}=\f{\gamma+1}{2\gamma}-\f{1}{n},$$ which further implies $\F(x,t)\in L^{2s}(0,T; L^{\alpha_1}(B_{4}^c))$. 

On the other hand if $|y|\le 2$, then $|x-y|\ge 2$. Since $\textrm{Supp}\psi\subset B_{2}\times[0,T]$, there follows 
\begin{equation*}
\begin{split}
|\F(x,t)|&\le C\int_{B_{2}}\f{1}{|x-y|^{n-1}}\r|\u|dy\le C\int_{B_{2}}\f{1}{|x-y|^{n-2}}\r|\u| dy\\
&\le C\int_{\R^n}\f{1}{|x-y|^{n-2}}\r|\u| dy=CI_{2}(\r|\u|).
\end{split}
\end{equation*}
Sobolev's embedding theorem implies $$I_{2}(\r\u)\in L^\infty(0,T; L^{\alpha_2}(\R^n)\quad\textrm{with}\quad \f{1}{\alpha_2}=\f{\gamma+1}{2\gamma}-\f{2}{n}.$$ which further implies $\F(x,t)\in L^{2s}(0,T; L^{\alpha_2}(B_{4}^c))$. As $\gamma\in (1,\f{n}{2}]$, one has $\alpha_1\le 2<\alpha_2$, and the desired estimate $\F(x,t)\in L^{2s}(0,T; L^\tau(B_{4}^c))$ follows from the interpolation. For $\gamma=1$, a similar interpolation argument works with Lebesgue spaces replaced by weak Lebesgue spaces.

\texttt{Case 2: $|x|\le 4$.} Since $\textrm{Supp}\psi\subset G_r$, one has $\r\le \r 1_{G_r}$ and
\begin{equation}\label{52}
|\F(x,t)|\le C\int_{\R^n}\f{1}{|x-y|^{n-1}}\psi\r|\u|(y,t)dy\le C\sqrt{\mathcal{L}(x,t)}\sqrt{I_1(\r 1_{G_r}) (x,t)},
\end{equation}
where
$$\mathcal{L}(x,t)=\int_{\R^n}|x-y|^{1-n}\psi\r|\u|^2(y)dy.$$
Since $\textrm{Supp}\psi\subset B_{2}\times[0,T]$, one has
$$\mathcal{L}(x,t)=\int_{B_{2}}(\psi\r|\u|^2)^{\f{n-1+\dl}{n}}|x-y|^{-(n-1)-\f{\dl}{n}}(\psi\r|\u|^2)^{\f{1-\dl}{n}}|x-y|^{\f{\dl}{n}}dy,$$
and hence
$$\mathcal{L}(x,t)\le C\int_{B_{2}}(\psi\r|\u|^2)^{\f{n-1+\dl}{n}}|x-y|^{-(n-1)-\f{\dl}{n}}(\psi\r|\u|^2)^{\f{1-\dl}{n}}dy\quad\textrm{for}\quad x\in B_{4}.$$
Applying the H\"{o}lder inequality gives
$$\mathcal{L}(x,t)\le C\left(\int_{B_{2}}\psi\r|\u|^2|x-y|^{\f{-n(n-1)-\dl}{n-1+\dl}}dy\right)^{\f{n-1+\dl}{n}}\left(\int_{B_{2}}\psi\r|\u|^2dy\right)^{\f{1-\dl}{n}}\quad\textrm{for}\quad x\in B_{4},$$
and thus there follows from Fubini's theorem that
\begin{equation}\label{53}
\begin{split}
\int_{B_{4}}\mathcal{L}(x,t)^{\f{n}{n-1+\dl}}dx&\le C\left(\int_{B_{2}}\psi\r|\u|^2dy\right)^{\f{1-\dl}{n-1+\dl}}\int_{B_{4}}\int_{B_{2}}\psi\r|\u|^2(y,t)|x-y|^{\f{-n(n-1)-\dl}{n-1+\dl}}dydx\\
&\le C\left(\int_{B_{2}}\psi\r|\u|^2dy\right)^{\f{n}{n-1+\dl}}\le C(\dl)\|\r|\u|^2\|_{L^1(\R^n)}^{\f{n}{n-1+\dl}}
\end{split}
\end{equation}
since $\int_{B_{4}}|x-y|^{\f{-n(n-1)-\dl}{n-1+\dl}}dx\le C(\dl)$.

Combining \eqref{52} and \eqref{53} yields the desired estimate as
\begin{equation*}
\begin{split}
\|\F\|_{L^{2s}(0,T; L^{\tau}(B_{4}))}&\le C(\dl) \|\mathcal{L}\|^{\f12}_{L^\infty(0,T; L^{\f{n}{n-1+\dl}}(B_{4}))}\|I_1(\r 1_{G_r})\|^{\f12}_{L^s(\R^n\times[0,T])}\\
&\le C(\dl) \|\r|\u|^2\|^{\f12}_{L^\infty(0,T; L^{1}(\R^n))}\|I_1(\r 1_{G_r})\|^{\f12}_{L^s(\R^n\times[0,T])}
\end{split}
\end{equation*}
with $\tau$ being given by the identity in \eqref{52a}. 
\end{proof}

Now we are ready to estimate the terms $I_1$ and $I_4$. Recall that $I_1$ has the form
\begin{equation}\label{57}
I_1=\int_0^T\int_{\R^n}\u_j \eqref{51a} dxdt+\int_0^T\int_{\R^n}\u_j \eqref{51b}dxdt.
\end{equation}
Thanks to the identity
$$\Dv(A\times B)=B\cdot\Cu A-A\cdot\Cu B,$$
there follows
\begin{equation*}
\begin{split}
&\Big(\psi\r\u+\nabla(-\D)^{-1}\Dv(\psi\r\u)\Big)\cdot\nabla(-\D)^{-1}\Dv(e_j\r^\lambda)\\
&\quad=(-\D)^{-1}\Cu\Cu(\psi\r\u)\cdot\nabla(-\D)^{-1}\Dv(e_j\r^\lambda)\\
&\quad=\Dv\Big((-\D)^{-1}\Cu(\psi\r\u)\times\nabla(-\D)^{-1}\Dv(e_j\r^\lambda)\Big).
\end{split}
\end{equation*}
This identity, combined with integration by parts, implies
\begin{equation*}
\begin{split}
&\left|\int_0^T\int_{\R^n}\u_j \eqref{51a} dxdt\right|\\
&\quad=\left|\int_0^T\int_{\R^n}\nabla\u_j\cdot \Big((-\D)^{-1}\Cu(\psi\r\u)\times\nabla(-\D)^{-1}\Dv(e_j\r^\lambda)\Big)dxdt\right|,
\end{split}
\end{equation*}
and hence using Lemma \ref{L51} and Lemma \ref{L52} one gets
\begin{equation}\label{58}
\begin{split}
\left|\int_0^T\int_{\R^n}\u_j \eqref{51a} dxdt\right|&\le C(T,s)\|\nabla\u\|_{L^2((0,T)\times\R^n)}\|\F\|_{L^{2s}(0,T; L^\tau(\R^n))}\|\r^\lambda\|_{L^\infty(0,T; L^{\f{2\tau}{\tau-2}}(\R^n))}\\
&<\infty
\end{split}
\end{equation}
as $2\tau\lambda/(\tau-2)\le \gamma$ due to \eqref{50a}.

Similarly due to
\begin{equation*}
\begin{split}
&\nabla(-\D)^{-1}\Dv(\psi\r\u)\cdot \Big(e_j\r^\lambda+\nabla(-\D)^{-1}\Dv(e_j\r^\lambda)\Big)\\
&\quad=\nabla(-\D)^{-1}\Dv(\psi\r\u)\cdot (-\D)^{-1}\Cu\Cu(e_j\r^\lambda)\\
&\quad=\Dv\left[(-\D)^{-1}\Dv(\psi\r\u)\cdot (-\D)^{-1}\Cu\Cu(e_j\r^\lambda)\right],
\end{split}
\end{equation*}
an integration by parts yields
\begin{equation*}
\begin{split}
&\left|\int_0^T\int_{\R^n}\u_j \eqref{51b} dxdt\right|\\
&\quad=\left|\int_0^T\int_{\R^n}\nabla\u_j\cdot\left[(-\D)^{-1}\Dv(\psi\r\u)\cdot(-\D)^{-1}\Cu\Cu(e_j\r^\lambda)\right]dxdt\right|,
\end{split}
\end{equation*}
and hence using Lemma \ref{L51} and Lemma \ref{L52} one gets
\begin{equation}\label{59}
\begin{split}
\left|\int_0^T\int_{\R^n}\u_j \eqref{51b} dxdt\right|&\le C(T,s)\|\nabla\u\|_{L^2((0,T)\times\R^n)}\|\F\|_{L^{2s}(0,T; L^\tau(\R^n))}\|\r^\lambda\|_{L^\infty(0,T; L^{\f{2\tau}{\tau-2}}(\R^n))}\\
&<\infty.
\end{split}
\end{equation}
Combining \eqref{57}, \eqref{58}, and \eqref{59}, one obtains
\begin{equation*}
|I_1|\le C(T,s)\|\nabla\u\|_{L^2((0,T)\times\R^n)}\|\F\|_{L^{2s}(0,T; L^\tau(\R^n))}\|\r^\lambda\|_{L^\infty(0,T; L^{\f{2\tau}{\tau-2}}(\R^n))}<\infty.
\end{equation*}

Finally for $I_4$, an integration by parts gives
\begin{equation*}
\begin{split}
I_4&=(\lambda-1)\int_0^T\int_{\R^n}\mathcal{A}_i[\psi\r\u_i]\r^\lambda \Dv\u dxdt\\
&=(\lambda-1)\int_0^T\F_1\r^\lambda \Dv\u dxdt,
\end{split}
\end{equation*}
and thus
\begin{equation*}
\begin{split}
|I_4|\le C\|\F\|_{L^{2s}(0,T; L^\tau(\R^n))}\|\r^\lambda\|_{L^\infty(0,T; L^{\f{2\tau}{\tau-2}}(\R^n))}\|\Dv\u\|_{L^2((0,T)\times\R^n)}<\infty
\end{split}
\end{equation*}
as desired.

\subsection{Proof of Theorem \ref{mt2}}

We finish the proof of Theorem \ref{mt2} as an application of Theorem \ref{4m}.
\begin{proof}[Proof of Theorem \ref{mt2}]
The \textit{a priori} estimate \eqref{energy1} implies that up to a subsequence as $\i\rightarrow 0$ there holds
$$\u^\i\rightarrow\u\quad\textrm{weakly in}\quad L^2(0,T; H^1(\R^n))$$
and
$$\r^\i\rightarrow \r\quad\textrm{weakly$^*$ in}\quad L^\infty(0,T; L^\gamma(\R^n)).$$

The improved integrability of the density in Proposition \ref{p3} implies
\begin{equation}\label{59}
\|\r^\i\|_{L_{loc}^{\gamma+\lambda}(G_r)}\le c
\end{equation}
for some constant $c$ and $\lambda$ in \eqref{50a}. The bound \eqref{59}, combining with \textit{a priori} estimate \eqref{energy1}, gives that $\r^\i\u^\i$ is uniformly bounded in 
\begin{equation*}
\|\r^\i\u^\i\|_{L^{2(\gamma+\lambda)}(0,T; L^{2(\gamma+\lambda)/(1+\gamma+\lambda)}(G_r\cap B_R))}\le \|\r^\i|\u^\i|^2\|^{\f12}_{L^\infty(0,T;L^1(\R^n))}\|\r^\i\|^{\f12}_{L^{\gamma+\lambda}(G_r\cap (B_R\times[0,T]))}.
\end{equation*}
Moreover the continuity equation gives us the uniform bound $\partial_t\r^\i$ in $L^1(0,T; W^{-L,p})$ for some large $L>0$ and $p\in (1,\infty)$. This, combining Lemma 5.1 in \cite{PL} and the uniform bound of $\nabla\u^\i\in L^2(\R^n\times(0,T))$, implies that $\r^\i\u^\i$ converges to $\r\u$ at least in the sense of distributions in the set $G_r$. Moreover, with the help of the decay rate of the density on $G_r$ in Theorem \ref{4m}, a similar argument as that in Proposition \ref{p6} yields the convergence of
\begin{equation}\label{50ac}
\r^\i\u^\i\otimes\u^\i\rightarrow \r\u\otimes\u \quad\textrm{in}\quad G_r
\end{equation}
as $\i\rightarrow 0$ in the sense of distributions. The convergence \eqref{50ac} implies that the reduced defect measure $\theta$ is supported in $G_r^c$ for all $r>0$; that is
$$E=\textrm{Supp}(\theta)\subset G_r^c\quad\textrm{for all}\quad r>0.$$
Therefore there follows further that
$$\mathcal{H}^{\Gamma(n)+\alpha, 1+\beta}(E)\le \mathcal{H}^{\Gamma(n)+\alpha, 1+\beta}(G_r^c)\le c_1(r)\rightarrow 0\quad\textrm{as}\quad r\rightarrow 0,$$
and hence
$$\mathcal{H}^{\Gamma(n)+\alpha, 1+\beta}(E)=0.$$ 

The strong convergence of $\r^\i\rightarrow \r$ in $L^1(E^c)$ follows from the classical convexity argument in \cite{PL, F, NP}. Indeed, in $G_r$, we have now
\begin{equation}\label{510aa}
\begin{cases}
\partial_t\r+\Dv(\r\u)=0\\
\partial_t(\r\u)+\Dv(\r\u\otimes\u)-\mu\D\u-\xi\nabla\Dv\u+\nabla \ov{\r^\gamma}=0
\end{cases}
\end{equation}
in the sense of distributions, where $\ov{f}$ means the weak $L^1$ limit of $f^\i$ as $\i\rightarrow 0$. 

Multiplying the second equation of \eqref{510aa} by $\psi \mathcal{A}_i(\ov{\r^\lambda})$ yields a similar identity as \eqref{pes} with $(\r^\gamma,\r^\lambda)$ replaced by $(\ov{\r^\gamma}, \ov{\r^\lambda})$. Moreover the identity \eqref{pes} holds true with $(\r^\gamma,\r^\lambda)$ replaced by $((\r^\i)^\gamma, (\r^\i)^\lambda)$.
Subtracting the resulting identities for $(\ov{\r^\gamma}, \ov{\r^\lambda})$ and $((\r^\i)^\gamma, (\r^\i)^\lambda)$ yields
\begin{equation}\label{510aaa}
\begin{split}
&\int_{\R^n}(\ov{\r^\gamma}-(\mu+\xi)\Dv\u)\ov{\r^\lambda} \psi dxdt\\
&\quad=
\lim_{\i\rightarrow 0}\left\{\int_{\R^n}((\r^\i)^\gamma-(\mu+\xi)\Dv\u^\i)(\r^\i)^\lambda \psi dxdt+\sum_{i=1}^4 I_i+\sum_{j=1}^5 J_j\right\},
\end{split}
\end{equation}
where $\psi$ is the test function as in \eqref{pes}, and 
\begin{equation*}
\begin{split}
I_1&=\int_{\R^n}\Big[\r\u_i\u_j \partial_j\mathcal{A}_i(\ov{\r^\lambda})-\r\u_i\partial_j\mathcal{A}_i[\u_j\ov{\r^{\lambda}}] \Big] \psi dxdt\\
&\quad-\int_{\R^n}\Big[\r^\i\u^\i_i\u^\i_j \partial_j\mathcal{A}_i((\r^\i)^\lambda)-\r^\i\u^\i_i\partial_j\mathcal{A}_i[\u^\i_j(\r^\i)^{\lambda}] \Big] \psi dxdt\\
I_2&=-\mu\int_{\R^n}\partial_j\u \partial_j\mathcal{A}_i(\ov{\r^\lambda}) \psi dxdt+\mu\int_{G_r}\partial_j\u^\i \partial_j\mathcal{A}_i((\r^\i)^\lambda) \psi dxdt\\
I_3&=-\xi\int_{\R^n}\Dv\u \partial_i\mathcal{A}_i(\ov{\r^\lambda})\psi dxdt+\xi\int_{G_r}\Dv\u^\i \partial_i\mathcal{A}_i((\r^\i)^\lambda) \psi dxdt\\
I_4&=-(\lambda-1)\int_{\R^n}\r\u_i\mathcal{A}_i\left[\ov{\r^\lambda}\Dv\u\right] \psi dxdt+(\lambda-1)\int_{\R^n}\r^\i\u^\i_i\mathcal{A}_i\left[(\r^\i)^\lambda\Dv\u^\i\right] \psi dxdt,
\end{split}
\end{equation*}
and
\begin{equation*}
\begin{split}
J_1&=\int_0^T\int_{\R^n}\partial_j\psi \mathcal{A}_j(\ov{\r^\lambda})\ov{\r^\gamma} dxdt-\int_0^T\int_{\R^n}\partial_j\psi \mathcal{A}_j((\r^\i)^\lambda)(\r^\i)^\gamma dxdt\\
J_2&=\int_0^T\int_{\R^n}\r\u_i\u_j \mathcal{A}_j(\ov{\r^\lambda})\partial_i\psi dxdt-\int_0^T\int_{\R^n}\r^\i\u^\i_i\u^\i_j \mathcal{A}_j((\r^\i)^\lambda)\partial_i\psi dxdt\\
J_3&=\int_0^T\int_{\R^n}\partial_t\psi \r\u_j\mathcal{A}_j(\ov{\r^\lambda})dxdt-\int_0^T\int_{\R^n}\partial_t\psi \r^\i\u^\i_j\mathcal{A}_j((\r^\i)^\lambda)dxdt\\
J_4&=-\mu\int_0^T\int_{\R^n}\partial_j\u_i \mathcal{A}_i(\ov{\r^\lambda}) \partial_j\psi dxdt+\mu\int_0^T\int_{\R^n}\partial_j\u_i^\i \mathcal{A}_i((\r^\i)^\lambda) \partial_j\psi dxdt\\
J_5&=-\xi\int_0^T\int_{\R^n}\Dv\u \mathcal{A}_i(\r^\lambda)\partial_i\psi dxdt+\xi\int_0^T\int_{\R^n}\Dv\u^\i \mathcal{A}_i((\r^\i)^\lambda)\partial_i\psi dxdt.
\end{split}
\end{equation*}
The convergence of $I_1\rightarrow 0$ follows from the same strategy in subsections 5.2 and 5.3 using the identity \eqref{57} with the aid of the function $\F$ and Lemma 6.1 in \cite{F} (see also Lemma 4.25 in \cite{NP}); while the convergence of $I_i$ ($2\le i\le 4$) and $J_j$ ($1\le j\le 5$) to zero can be verified as those in Proposition \ref{p6}. 
Therefore letting $\i\rightarrow 0$ and suitably choosing the function $\psi$ in \eqref{510aaa} recover the weak continuity of \textit{effective viscous fluxes}, $\r^\gamma-(\lambda+\mu)\Dv\u$, in the compact subset $K\Subset G_r$ as
$$\int_{K}(\ov{\r^\gamma}-(\lambda+\mu)\Dv\u)\r^\lambda dxdt=
\lim_{\i\rightarrow 0}\int_{K}((\r^\i)^\gamma-(\lambda+\mu)\Dv\u^\i)(\r^\i)^\lambda dxdt.$$
The weak continuity above, combining the convexity argument and the concept of reformalized solutions, yields the strong convergence of the density in $K\Subset G_r$ (see \cite{PL, F, NP}); that is
$$\r^\i\rightarrow \r\quad\textrm{strongly in}\quad L^p(K)\quad\textrm{for all}\quad p< \gamma+\lambda$$
as $\i\rightarrow 0$. Therefore the limit function $(\r,\u)$ is a weak solution of the system \eqref{e1} in the sense of distributions in the set $E^c$.
\end{proof}

\subsection{Proof of Theorem \ref{mt1}}

We finish the proof of Theorem \ref{mt1} when the spherical Hausdorff dimension of the concentration set is less than one.

\begin{proof}[Proof of Theorem \ref{mt1}]
For $F\subset \R^n\times\R^+$, denote by $\pi$ the projection onto the time axis
$$\pi F=\{t\in \R: \exists (x,t)\in F\}.$$
By hypothesis there exists a number $\alpha$ with $0<\alpha<1$ and a family of closed set $\{F_r\}$ such that
\begin{equation}\label{30g1}
\theta(F_r)=0\quad \forall r>0,
\end{equation}
and
\begin{equation}\label{30g2}
\mathcal{H}^\alpha_r(F_r^c)\le C
\end{equation}
for some $C$ independent of $\alpha$ and $r$.

The identity \eqref{30g1} implies that
\begin{equation}\label{30g4}
\begin{split}
\lim_{\i\rightarrow 0}\int_0^T\int_{B_R} \r^\i|\u^\i-\u|^2dxdt&=\lim_{\i\rightarrow 0}\int\int_{F_r}\r^\i|\u^\i-\u|^2dxdt+\lim_{\i\rightarrow 0}\int\int_{F^c_r}\r^\i|\u^\i-\u|^2dxdt\\
&=\theta(F_r)+\lim_{\i\rightarrow 0}\int\int_{F^c_r}\r^\i|\u^\i-\u|^2dxdt\\
&=\lim_{\i\rightarrow 0}\int\int_{F^c_r}\r^\i|\u^\i-\u|^2dxdt.
\end{split}
\end{equation}
The bound on Hausdorff premeasure \eqref{30g2} implies that there exists a countable covering $\{B_{r_i}(x_i, t_i)\}$ of $F_r^c$ with $r_i\le r$ such that
$$F_r^c\subset \bigcup B_{r_i}(x_i,t_i),\qquad r_i\le r,\qquad \sum r_i^\alpha\le 2\mathcal{H}_r^\alpha(F_r^c)\le 2C.$$
Thus one has $\pi F_r^c\subset \cup \pi B_{r_i}(x_i,t_i)\subset\R^+$ and
\begin{equation}\label{30g3}
\begin{split}
m_1(\pi F_r^c)&\le m_1\Big(\bigcup \pi B_{r_i}(x_i,t_i)\Big)\le \sum m_1(\pi B_{r_i}(x_i,t_i))\\
&\le \sum r_i\le r^{1-\alpha}\sum r_i^\alpha\le 2C r^{1-\alpha}
\end{split}
\end{equation}
for $0<\alpha<1$.
This estimate, combined with the \textit{a priori} kinetic energy
$$\max_{0\le t\le T}\int_{B_R}\r^\i|\u^\i-\u|^2dx\le C,$$
implies
\begin{equation*}
\begin{split}
\int\int_{F_r^c}\r^\i|\u^\i-\u|^2dxdt&\le \int_{\pi F_r^c}\left(\int_{B_R}\r^\i|\u^\i-\u|^2dx\right) dt\\
&\le C(T,R)m(\pi F_r^c)\le C(T,R) r^{1-\alpha}.
\end{split}
\end{equation*}
This, combined with the identity \eqref{30g4}, yields
$$\lim_{\i\rightarrow 0}\int_0^T\int_{B_R}\r^\i|\u^\i-\u|^2dxdt\le C(T,R)r^{1-\alpha}\quad \textrm{with}\quad 0<\alpha<1$$
Because $r$ is arbitrary the left-hand side is zero. Hence we have strong convergence
$$\sqrt{\r^\i}(\u^\i-\u)\rightarrow 0 \quad\textrm{in}\quad L^2(B_R\times(0,T)).$$

\end{proof}

\bigskip\bigskip


\section{Appendix}
This appendix provides a detailed verification for the convolution of $|x|^{\f{p(1-n)}{2(p-1)}}$ in the proof of Proposition \ref{p2}.

\begin{Lemma}\label{la1}
Let $\f{2n}{n+1}<p\le 2$ and $g(x)=|x|^{\f{p(1-n)}{2(p-1)}}$ for $|x|\le s\le 1/2$ and $0$ otherwise. Then $\supp (g\star g)\subset \{x\in \R^n: |x|\le 2s\}$ and
\begin{equation*}
g\star g(x)\le \ov{G}(x)\equiv C(p,n)\begin{cases}
G(x)\quad&\textrm{if}\quad p=2\\
|x|^{\f{p-n}{p-1}}\quad&\textrm{if}\quad p\in (\f{2n}{n+1},2),
\end{cases}
\end{equation*}
where $G(x)$ is the fundamental solution of $-\D$ in $\R^n$.
\end{Lemma}
\begin{proof}
Since $\supp g(x)=\{x:|x|\le s\}$, it follows 
$$g\star g(x)=\int_{\R^n}g(x-y)g(y)dy=\int_{|y|\le s}g(x-y)g(y)dy,$$
and hence $\supp (g\star g)\subset \{x: |x|\le 2s\}$.

\texttt{Case 1: $s/2\le |x|\le 2s$.} In this case, one has
\begin{equation*}
\begin{split}
g\star g(x)&=\int_{|y|\le s/4}g(x-y)g(y)dy+\int_{s/4\le |y|\le s}g(x-y)g(y)dy\\
&\le \int_{|y|\le s/4}|x-y|^{\f{p(1-n)}{2(p-1)}}|y|^{\f{p(1-n)}{2(p-1)}}dy+\int_{s/4\le |y|\le s}|x-y|^{\f{p(1-n)}{2(p-1)}}|y|^{\f{p(1-n)}{2(p-1)}}dy\\
&\le (s/4)^{\f{p(1-n)}{2(p-1)}}\int_{|y|\le s/4}|y|^{\f{p(1-n)}{2(p-1)}}dy+(s/4)^{\f{p(1-n)}{2(p-1)}}\int_{s/2\le |y|\le s}|x-y|^{\f{p(1-n)}{2(p-1)}}dy\\
&\le \begin{cases}
C\quad &\textrm{if}\quad p=2\quad\textrm{and}\quad n=2\\
C(p) s^{{\f{p-n}{p-1}}}\quad &\textrm{otherwise}
\end{cases}\\
&\le C(p,n)\begin{cases}
G(x)\quad&\textrm{if}\quad p=2\\
|x|^{\f{p-n}{p-1}}\quad&\textrm{if}\quad p\in (\f{2n}{n+1},2)
\end{cases}
\end{split}
\end{equation*}
since $|x|\le \f12$ and $p>\f{2n}{n+1}$.

\texttt{Case 2: $|x|\le s/2$.} Since $|x|\le s/2$, it follows that $\{y: |y|\le |x|/2\}\cup \{y:|x-y|\le |x|/2\}\subset \{y:|y|\le s\}.$ Set $A=\Big(\{y: |y|\le |x|/2\}\cup \{y:|x-y|\le |x|/2\}\Big)^c\cap \{y:|y|\le s\}$.  Note $A\subset \{|y|\le s: |y|\ge |x|/2\}$ and $A\subset\{|y|\le s: |x-y|\ge |x|/2\}$. Therefore
\begin{equation*}
\begin{split}
g\star g(x)&=\int_{|y|\le |x|/2}g(x-y)g(y)dy+\int_{|x-y|\le |x|/2}g(x-y)g(y)dy+\int_{A}g(x-y)g(y)dy\\
&\le \int_{|y|\le |x|/2}|x-y|^{\f{p(1-n)}{2(p-1)}}|y|^{\f{p(1-n)}{2(p-1)}}dy+\int_{|x-y|\le |x|/2}|x-y|^{\f{p(1-n)}{2(p-1)}}|y|^{\f{p(1-n)}{2(p-1)}}dy\\
&\quad+\int_{A}|x-y|^{\f{p(1-n)}{2(p-1)}}|y|^{\f{p(1-n)}{2(p-1)}}dy\\
&\le (|x|/2)^{\f{p(1-n)}{2(p-1)}}\int_{|y|\le |x|/2}|y|^{\f{p(1-n)}{2(p-1)}}dy+(|x|/2)^{\f{p(1-n)}{2(p-1)}}\int_{|x-y|\le |x|/2}|x-y|^{\f{p(1-n)}{2(p-1)}}dy\\
&\quad+\int_{\{|y|\le s:|x-y|\ge |x|/2\}}|x-y|^{\f{p(1-n)}{p-1}}dy+\int_{\{|y|\le s:|y|\ge |x|/2\}}|y|^{\f{p(1-n)}{p-1}}dy\\
&\le  \begin{cases}
C &\textrm{if}\quad p=2\quad\textrm{and}\quad n=2\\
C(p,n) |x|^{\f{p-n}{p-1}}\quad &\textrm{otherwise}
\end{cases}\\
&\le C(p,n)\begin{cases}
G(x)\quad&\textrm{if}\quad p=2\\
|x|^{\f{p-n}{p-1}}\quad&\textrm{if}\quad p\in (\f{2n}{n+1},2).
\end{cases}
\end{split}
\end{equation*}
\end{proof}

Next we turn to the verification of the inequality \eqref{in}.
\begin{proof}[Proof of the inequality \eqref{in}]
First of all, since $\gamma(n)$ and $\Gamma(n)$ are decreasing with respect to $\gamma$, the function $n-\Gamma(n)$ is increasing with respect to $\gamma$. Since as $\gamma=1$, 
$n-\Gamma(n)$ equals to $2$ if $n=2$ and equals to $6/5$ if $n=3$, it follows that $n-\Gamma(n)>1$ as $\gamma\in [1,n/2]$.

Next from the definition of $\Gamma(n)$, it follows that $n-\Gamma(n)\le n\gamma/(\gamma(n)+1)$.  Moreover the inequality $\gamma(n)<n-1$ implies that 
\begin{equation*}
\begin{cases}
\gamma(n)+1> n-\gamma(n-1-\gamma(n))\quad\textrm{if}\quad \gamma>1\\
\gamma(n)+1=n-\gamma(n-1-\gamma(n))\quad\textrm{if}\quad\gamma=1,
\end{cases}
\end{equation*}
and hence there follows
\begin{equation*}
n-\Gamma(n)
\begin{cases}
\le\f{n\gamma}{\gamma(n)+1}<\f{n\gamma}{n-\gamma(n-1-\gamma(n))}\quad\textrm{if}\quad \gamma>1\\
=\f{n\gamma}{\gamma(n)+1}=\f{n\gamma}{n-\gamma(n-1-\gamma(n))}\quad\textrm{if}\quad \gamma=1
\end{cases}
\end{equation*}
as claimed.
\end{proof}

\bigskip\bigskip


\section*{Acknowledgement}
The work is partially supported by the ECS grant 9048035. The author would like to thank Antony Novotn\'{y} for pointing out a mistake in Lemma  \ref{la1} in the previous version.

\bigskip\bigskip


\end{document}